\documentclass[draft,reqno]{amsart}
\usepackage[american]{babel}
\usepackage{amssymb,euscript}
\numberwithin{equation}{section}
\newtheorem{theorem}{Theorem}[section]
\newtheorem{lemma}{Lemma}[section]
\newtheorem{corollary}{Corollary}[section]
\theoremstyle{definition}
\newtheorem{definition}{Definition}[section]
\theoremstyle{remark}
\newtheorem{remark}{Remark}[section]

\author[Balashov Faminskii]{Oleg S. Balashov, Andrei V. Faminskii}
\address{RUDN University, 6 Miklukho-Maklaya St, Moscow, 117198, Russian Federation}
\email{balashovos@s1238.ru, faminskiy-av@rudn.ru}

\title[Inverse problems for the generalized KdV equation ]{Inverse problems with integral conditions for the generalized Korteweg--de~Vries equation}

\date{}

\begin{document}
\maketitle

\begin{abstract}
Results on well-posedness of three inverse problems with integral conditions on a bounded interval for the generalized Korteweg--de~Vries equation without any restrictions on the growth rate of nonlinearity are established. Either the right-hand side of equation or the boundary data, or both are chosen as controls. The considered solutions are regular with respect to the time variable. Assumptions on smallness of the input data or smallness of a time interval are required.
\end{abstract}

\section{Introduction. Notation. Description of main results}\label{S1}

The paper is concerned with inverse problems for the generalized Korteweg--de~Vries equation
\begin{equation}\label{1.1}
u_t +b u_x + u_{xxx} + \bigl(g(u)\bigr)_x = f(t,x),
\end{equation}
posed on an interval $I = (0,R)$ ($R>0$ is arbitrary). For certain $T>0$ in a rectangle $Q_T = (0,T)\times I$ consider an initial-boundary value problem for equation \eqref{1.1} with an initial condition
\begin{equation}\label{1.2}
u(0,x) = u_0 (x),\quad x\in [0,R],
\end{equation}
and boundary conditions
\begin{equation}\label{1.3}
u(t,0) = \mu_0(t), \quad u(t,R) = \nu_0(t), \quad u_x(t,R) = \nu_1(t),\quad t\in [0,T].
\end{equation}
The main attention is paid to an inverse problem with two integral overdetermination conditions for $j=1 $ and $2$
\begin{equation}\label{1.4}
\int_I u(t,x) \omega_j(x)\, dx = \varphi_j(t),\quad t\in [0,T],
\end{equation}
for certain given functions $\omega_j$ and $\varphi_j$. Two other inverse problems with one overdetermination condition \eqref{1.4} for $j=0$ are also considered. 

In the first problem the function $f$ is represented as
\begin{equation}\label{1.5}
f(t,x) = h_0(t,x) + F(t) h(t,x),
\end{equation}
where the functions $h_0$ and $h$ are given, while the function $F$ is unknown. Moreover, the boundary data $\nu_1$ is also unknown and the goal is to determine the functions $F$ and $\nu_1$ such that the corresponding solution to problem \eqref{1.1}--\eqref{1.3} satisfies conditions \eqref{1.4}, $j=1,\, 2$. 

In the second problem the function $F$ in \eqref{1.5} is unknown, in the third one --- the function $\nu_1$. The goal is to determine each of this functions to provide for the solution of the same initial-boundary value problem implementation of one additional condition \eqref{1.4}, $j=0$. 

The Korteweg--de~Vries equation
$$
u_t + b u_x + u_{xxx} + uu_x = f(t,x)
$$
is the model equation for description of nonlinear waves in dispersive media. There is a lot of papers devoted to various initial-boundary value problems for this equation as well as for its generalization \eqref{1.1}, in particular, posed on a rectangle (see, for example, the bibliography in \cite{F19-1}). Inverse problems were mostly studied with the terminal overdetermination condition
$$
u(T,x) = u_T(x),
$$
which is usually called the controllability problems (again see, for example, the bibliography in \cite{F19-1}).

On the contrary, there is a few papers concerned with inverse problems for the Korteweg--de~Vries equation and its generalizations with integral overdetermination conditions. In \cite{F19-1} two inverse problems for the Korteweg--de~Vries equation itself posed on a rectangle $Q_T$ with initial and boundary conditions \eqref{1.2}, \eqref{1.3} under one integral condition \eqref{1.4} were studied. Either the function $F$ from the right-hand side of the equation\eqref{1.5} or the boundary data $\nu_1$ were unknown. Existence and uniqueness  results in the class of weak solutions under either small input data or a small time interval were obtained. 

The inverse problem with two integral conditions \eqref{1.4} for the Korteweg--de~Vries equation type equation
$$
u_t + u_{xxx} + uu_x + \alpha(t)u = F(t)h(x)
$$
in the periodic case, where the functions $\alpha$ and $F$ were unknown, was considered in \cite{LCL} and the existence and uniqueness results were established for a small time interval.

In \cite{C-FS} and \cite{Mar} similar to \cite{F19-1} results were obtained for the Kawahara equation which is the fifth-order generalization of the Korteweg--de~Vries equation (a more general nonlinearity $g'(u)u_x$ was considered in the second paper). 

Three inverse problems similar to the ones in the present paper were studied in \cite{FM} for the higher order nonlinear Schr\"odinger equation and results on well-posedness (that is existence, uniqueness and continuous dependence on the input data) in the classes of weak solutions either under small input data or a small time interval were obtained. Similar results for a wide class of odd-order quasilinear evolution equations and systems  with general nonlinearity were recently established in \cite{F23, BF24, BF25}.

Note that in all cited papers, where the general nonlinearity were considered (\cite{Mar, FM, F23, BF24, BF25}), this nonlinearity was subjected to certain growth restrictions. For example, it follows from results of paper \cite{F23} that for equation \eqref{1.1} it was assumed that for all $u\in \mathbb R$
$$
|g'(u)| \leq c|u|^a, \quad a\in (0,2].
$$ 

The inverse problem with unknown $F$ and one integral overdetermination condition on unbounded intervals for the Korteweg--de~Vries equation itself was studied in \cite{F19-2}.

The significance and various applications of integral conditions of type \eqref{1.4} were discussed in \cite{POV} (see also \cite{FJ, FN}). 

In the present paper the inverse problems are considered in function spaces regular with respect to the time variable. This approach gives an opportunity to avoid any restrictions on the growth rate of nonlinearity $g(u)$ in equation \eqref{1.1}. However, on this way certain accuracy with compatibility of initial and boundary data is required. Results on well-posedness are established under assumptions on smallness either of the input data or the time interval.

Fix for the whole paper certain natural $k$. Solutions of the considered problems are constructed in a special function space
\begin{multline*}
X^k(Q_T) = \{ u(t,x): \partial _t^m u \in C([0,T]; L_2(I)) \cap L_2(0,T; H^1(I)),\ 0\leq m \leq k; \\ \partial_t^m u \big|_{t=0} \in H^{3(k-m)}(I),\ 0\leq m \leq k-1\}
\end{multline*}
endowed with the natural 	norm
\begin{multline}\label{1.6}
\|u\|_{X^k(Q_T)} \\= \sum\limits_{m=0}^k\left[\max\limits_{t\in [0,T]} \|\partial_t^m u(t,\cdot)\|_{L_2(I)} + \|\partial_t^m u_x\|_{L_2(Q_T)}\right] +
\sum\limits_{m=0}^{k-1} \|\partial_t^m u\big|_{t=0} \|_{H^{3(k-m)}(I)}.
\end{multline}

\begin{remark}\label{R1.1}
In paper \cite{F19-1} weak solutions were constructed in the space
$$
X(Q_T) = C([0,T];L_2(I)) \cap L_2(0,T;H^1(I)).
$$
\end{remark}

For description of properties of right-hand side of the equation introduce a function space
\begin{multline*}
M^k(Q_T) = \{ f(t,x): \partial_t^m f \in L_2(Q_T),\ 0\leq m \leq k;\\
\partial_t^m f\big|_{t=0} \in H^{3(k-m-1)}(I), \ 0\leq m \leq k-1\}
\end{multline*}
endowed with the natural norm
\begin{equation}\label{1.7}
\|f\|_{M^k(Q_T)}  \\= \sum\limits_{m=0}^k \|\partial_t^m f\|_{L_2(Q_T)}+ \sum\limits_{m=0}^{k-1} \|\partial_t^m f\big|_{t=0} \|_{H^{3(k-m-1)}(I)}.
\end{equation}

\begin{remark}\label{R1.2}
It is obviously assumed that any function $f\in M^k(Q_T)$ is already changed on a zero measure set providing $\partial_t^m f \in C([0,T]; L_2(I))$, $0\leq m \leq k-1$.
\end{remark}

It is always assumed that the function $g$ satisfies the following condition:
\begin{equation}\label{1.8}
g\in C^{3k-1}(\mathbb R),\quad g(0) = g'(0) =0.
\end{equation}

\begin{remark}\label{R1.3}
There is no loss of generality in the the assumption $g(0) = g'(0) = 0$ since nonlinearity in equation \eqref{1.1} can be written in a form $g'(u)u_x$ and this equation contains the term $bu_x$. 
\end{remark}

Formally differentiating $g(u(t,x))$, one can easily derive an equality
\begin{equation}\label{1.9}
\partial_t^m g(u) = \sum\limits_{l=0}^m g^{(l)}(u) \sum\limits_{n_1+\dots+n_l =m} c_{m,l,n_1,\dots,n_l}\partial_t^{n_1} u \dots \partial_t^{n_l} u
\end{equation}
(where $c_{0,0}=1$ and $c_{m,0} =0$ for $m\geq 1$).

In order to define compatibility conditions introduce the following auxiliary functions $\Phi_m(x) = \Phi_m(x;u_0,f)$: let $\Phi_0 \equiv u_0$ and for $m\geq 1$
\begin{multline}\label{1.10}
\Phi_m \equiv \partial_t^{m-1} f\big|_{t=0} - (b\Phi'_{m-1} + \Phi'''_{m-1}) \\ -
\Bigl(\sum\limits_{l=0}^{m-1} g^{(l)}(u_0) \sum\limits_{n_1+\dots+n_l =m-1} c_{m-1,l,n_1,\dots,n_l} \Phi_{n_1} \dots \Phi_{n_l} \Bigr)'.
\end{multline}

For description of properties of the boundary data introduce the fractional Sobolev spaces $H^s(0,T)$, $s\in \mathbb R$, which are defined as the spaces of restrictions on the interval $(0,T)$  
of functions in $H^s(\mathbb R)$, where
$$
H^s(\mathbb R) = \{f(t): (1+\xi^2)^{s/2} \widehat f(\xi) \in L_2(\mathbb R)\}
$$
($\widehat f$ is the Fourier transform of the function $f$). 

Regarding the functions $\omega_j$ we always need the following conditions:
\begin{equation}\label{1.11}
\omega_j \in H^{3}(I),\quad \omega_j(0)= \omega'_j(0) = \omega_j(R) =0.
\end{equation}
Define the functions
\begin{equation}\label{1.12}
\psi_j(t) \equiv \int_I h(t,x)\omega_j(x)\,dx, 
\end{equation}
where $h$ is the function from equality \eqref{1.5}.

Introduce the function space
$$
\widetilde H^{k}(0,T) = \{\varphi\in H^k(0,T): \varphi^{(m)}(0) =0, \ 0\leq m \leq k-1\}.
$$
Obviously, $\|\varphi^{(k)}\|_{L_2(0,T)}$ is the equivalent norm in $\widetilde H^{k}(0,T)$ and is used further.

Now we can pass to the main result for the first inverse problem.

\begin{theorem}\label{T1.1}
Let 
\begin{itemize}
\item
the function $g$ satisfy condition \eqref{1.8}, $u_0 \in H^{3k}(I)$, $\mu_0, \nu_0 \in H^{k+1/3}(0,T)$, $h_0 \in M^k(Q_T)$ and
$\mu_0^{(m)}(0) = \Phi_{m}(0;u_0,h_0)$, $\nu_0^{(m)}(0) = \Phi_m(R;u_0,h_0)$ for $0\leq m\leq k-1$; 
\item
for $j=1, \, 2$ the functions $\omega_j$ satisfy condition \eqref{1.11}, 
$\varphi_j \in H^{k+1}(0,T)$  and
\begin{equation}\label{1.13}
\varphi_j^{(m)}(0) = \int_I \Phi_m(x;u_0,h_0) \omega_j(x) \,dx, \quad 0\leq m \leq k;
\end{equation}
\item
$h\in C^k([0,T]; L_2(I))$ and
\begin{equation}\label{1.14}
\Delta(t) \equiv 
\begin{vmatrix}
\psi_1(t) & \omega_1'(R)\\
\psi_2(t) & \omega_2'(R)
\end{vmatrix} 
\ne 0 \quad \forall t\in [0,T];
\end{equation}
\item either an arbitrary $\delta>0$ or an arbitrary $T_0>0$ is fixed.
\end{itemize}
Denote
\begin{multline}\label{1.15}
c_1 = c_1(u_0,\mu_0,\nu_0,h_0,\varphi_1,\varphi_2) \equiv \|u_0\|_{H^{3k}(I)} + \|\mu_0\|_{H^{k+1/3}(0,T)} + \|\nu_0\|_{H^{k+1/3}(0,T)} \\+
\|h_0\|_{M^k(Q_T)} + \|\varphi_1^{(k+1)}\|_{L_2(0,T)} + \|\varphi_2^{(k+1)}\|_{L_2(0,T)} .
\end{multline}
Then either for the fixed $\delta>0$ there exists $T_0>0$ or for the fixed $T_0>0$ there exists $\delta>0$ such that in both cases if $c_1\leq \delta$ and $T\in (0,T_0]$, there exist a unique pair of functions $(F, \nu_1)$, satisfying
$F\in \widetilde H^k(0,T)$, $\nu_1 \in H^k(0,T)$, $\nu_1^{(m)}(0) = \Phi_m'(R;u_0,h_0)$ for $0\leq m \leq k-1$, and the corresponding unique solution of problem \eqref{1.1}--\eqref{1.3} $u\in X^k(Q_T)$ satisfying conditions \eqref{1.4} for $j=1, \, 2$, where the function $f$ is given by formula \eqref{1.5}.
Moreover, the map
\begin{equation}\label{1.16}
\bigl(u_0, \mu_0, \nu_0, h_0, \varphi^{(k+1)}_1,\varphi^{(k+1)}_2 \bigr) \to (u, F,\nu_1)
\end{equation}
is Lipschitz continuous on the corresponding subset of the closed ball of the radius $\delta$ in the space $H^{3k}(I) \times \bigl(H^{k+1/3}(0,T)\bigr)^2 \times M^k(Q_T) \times \bigl(L_2(0,T)\bigr)^2$ into the space $X^k(Q_T)\times \bigl(H^k(0,T)\bigr)^2$.
\end{theorem}

\begin{remark}\label{R1.4}
Because of the lack of smoothness with respect to $x$ for functions from the space $X^k(Q_T)$, solutions of problem \eqref{1.1}--\eqref{1.3} from this space are understood in the usual sense of the corresponding integral identity. The corresponding subset in the hypothesis of the theorem means the verification of the compatibility conditions.
\end{remark}

\begin{remark}\label{R1.5}
The smoothness properties $\mu_0,\nu_0 \in H^{k+1/3}(0,T)$, $\nu_1\in H^k(0,T)$ for the boundary data are natural, since if one considers the initial value linear problem
$$
v_t + v_{xxx} =0,\quad v\big|_{t=0} = v_0(x) \in H^{3k}(\mathbb R),
$$
then, by \cite{KPV}, its solution $v\in C(\mathbb R; H^{3k}(\mathbb R))$ (which can be easily constructed via the Fourier transform) for any $x\in \mathbb R$ satisfies the following relations: 
$$
\|D^{1/3} v(\cdot,x)\|_{H^k(\mathbb R)} = \|v_x(\cdot,x)\|_{H^k(\mathbb R)} = c\|v_0\|_{H^{3k}(\mathbb R)}.
$$
\end{remark}

Now we formulate the main result for the second inverse problem.

\begin{theorem}\label{T1.2}
Let 
\begin{itemize}
\item
the function $g$ satisfy condition \eqref{1.8}, $u_0 \in H^{3k}(I)$, $\mu_0, \nu_0 \in H^{k+1/3}(0,T)$, $\nu_1\in H^k(0,T)$, $h_0 \in M^k(Q_T)$, and $\mu_0^{(m)}(0) = \Phi_{m}(0;u_0,h_0)$, $\nu_0^{(m)}(0) = \Phi_m(R;u_0,h_0)$, $\nu_1^{(m)}(0) =  \Phi_m'(R;u_0,h_0)$ for $0\leq m\leq k-1$; 
\item
the function $\omega_0$ satisfy condition \eqref{1.11}, $\varphi_0 \in H^{k+1}(0,T)$  and
\begin{equation}\label{1.17}
\varphi_0^{(m)}(0) = \int_I \Phi_m(x;u_0,h_0) \omega_0(x) \,dx, \quad 0\leq m \leq k;
\end{equation}
\item
$h\in C^k([0,T]; L_2(I))$ and
\begin{equation}\label{1.18}
\psi_0(t) \ne 0 \quad \forall t\in [0,T];
\end{equation}
\item either an arbitrary $\delta>0$ or an arbitrary $T_0>0$ is fixed.
\end{itemize}
Denote
\begin{multline}\label{1.19}
c_2 = c_2(u_0,\mu_0,\nu_0,\nu_1,h_0,\varphi_0) \equiv \|u_0\|_{H^{3k}(I)} + \|\mu_0\|_{H^{k+1/3}(0,T)} + \|\nu_0\|_{H^{k+1/3}(0,T)} \\+ \|\nu_1\|_{H^k(0,T)}+
\|h_0\|_{M^k(Q_T)} + \|\varphi_0^{(k+1)}\|_{L_2(0,T)}.
\end{multline}
Then either for the fixed $\delta>0$ there exists $T_0>0$ or for the fixed $T_0>0$ there exists $\delta>0$ such that in both cases if $c_2\leq \delta$ and $T\in (0,T_0]$, there exist a unique  function $F \in \widetilde H^k(0,T)$ and the corresponding unique solution of problem \eqref{1.1}--\eqref{1.3} $u\in X^k(Q_T)$ satisfying condition \eqref{1.4} for $j=0$, where the function $f$ is given by formula \eqref{1.5}.
Moreover, the map
\begin{equation}\label{1.20}
\bigl(u_0, \mu_0, \nu_0,\nu_1, h_0, \varphi^{(k+1)}_0 \bigr) \to (u, F)
\end{equation}
is Lipschitz continuous on the corresponding subset of the closed ball of the radius $\delta$ in the space $H^{3k}(I) \times \bigl(H^{k+1/3}(0,T)\bigr)^2 \times H^k(0,T) \times M^k(Q_T) \times L_2(0,T)$ into the space $X^k(Q_T)\times H^k(0,T)$.
\end{theorem}

For the third inverse problem the result is the following.

\begin{theorem}\label{T1.3}
Let 
\begin{itemize}
\item
the function $g$ satisfy condition \eqref{1.8}, $u_0 \in H^{3k}(I)$, $\mu_0, \nu_0 \in H^{k+1/3}(0,T)$, $f \in M^k(Q_T)$ and $\mu_0^{(m)}(0) = \Phi_{m}(0;u_0,f)$, $\nu_0^{(m)}(0) = \Phi_m(R;u_0,f)$ for $0\leq m\leq k-1$; 
\item
the function $\omega_0$ satisfy condition \eqref{1.11}, $\varphi_0 \in H^{k+1}(0,T)$  and equalities \eqref{1.17} hold for $h_0\equiv f$;
\item
\begin{equation}\label{1.21}
\omega'_0(R) \ne 0;
\end{equation}
\item either an arbitrary $\delta>0$ or an arbitrary $T_0>0$ is fixed.
\end{itemize}
Denote
\begin{multline}\label{1.22}
c_3 = c_3(u_0,\mu_0,\nu_0,f,\varphi_0) \equiv \|u_0\|_{H^{3k}(I)} + \|\mu_0\|_{H^{k+1/3}(0,T)} + \|\nu_0\|_{H^{k+1/3}(0,T)} \\+ 
\|f\|_{M^k(Q_T)} + \|\varphi_0^{(k+1)}\|_{L_2(0,T)}.
\end{multline}
Then either for the fixed $\delta>0$ there exists $T_0>0$ or for the fixed $T_0>0$ there exists $\delta>0$ such that in both cases if $c_3\leq \delta$ and $T\in (0,T_0]$, there exist a unique  function $\nu_1 \in H^k(0,T)$, satisfying $\nu_1^{(m)}(0) = \Phi_m'(R;u_0,h_0)$ for $0\leq m \leq k-1$, and the corresponding unique solution of problem \eqref{1.1}--\eqref{1.3} $u\in X^k(Q_T)$ satisfying condition \eqref{1.4} for $j=0$.
Moreover, the map
\begin{equation}\label{1.23}
\bigl(u_0, \mu_0, \nu_0, f, \varphi^{(k+1)}_0 \bigr) \to (u, \nu_1)
\end{equation}
is Lipschitz continuous on the corresponding subset of the closed ball of the radius $\delta$ in the space $H^{3k}(I) \times \bigl(H^{k+1/3}(0,T)\bigr)^2 \times M^k(Q_T) \times L_2(0,T)$ into the space $X^k(Q_T)\times H^k(0,T)$.
\end{theorem}

The methods, developed for the inverse problems, can be also applied for the direct initial-boundary value problem itself.

\begin{theorem}\label{T1.4}
Let 
\begin{itemize}
\item
the function $g$ satisfy condition \eqref{1.8}, $u_0 \in H^{3k}(I)$, $\mu_0, \nu_0 \in H^{k+1/3}(0,T)$, $\nu_1\in H^k(0,T)$, $f \in M^k(Q_T)$, $\mu_0^{(m)}(0) = \Phi_{m}(0;u_0,f)$, $\nu_0^{(m)}(0) = \Phi_m(R;u_0,f)$, $\nu_1^{(m)}(0) =  \Phi_m'(R;u_0,f)$ for $0\leq m\leq k-1$; 
\item
either an arbitrary $\delta>0$ or an arbitrary $T_0>0$ is fixed.
\end{itemize}
Denote
\begin{multline}\label{1.24}
c_4 = c_4(u_0,\mu_0,\nu_0,\nu_1,f) \equiv \|u_0\|_{H^{3k}(I)} + \|\mu_0\|_{H^{k+1/3}(0,T)} + \|\nu_0\|_{H^{k+1/3}(0,T)} \\+ \|\nu_1\|_{H^k(0,T)}+
\|f\|_{M^k(Q_T)}.
\end{multline}
Then either for the fixed $\delta>0$ there exists $T_0>0$ or for the fixed $T_0>0$ there exists $\delta>0$ such that in both cases if $c_4\leq \delta$ and $T\in (0,T_0]$, there exists a    unique solution of problem \eqref{1.1}--\eqref{1.3} $u\in X^k(Q_T)$.
Moreover, the map
\begin{equation}\label{1.25}
\bigl(u_0, \mu_0, \nu_0,\nu_1,f\bigr) \to u
\end{equation}
is Lipschitz continuous on the corresponding subset of the closed ball of the radius $\delta$ in the space $H^{3k}(I) \times \bigl(H^{k+1/3}(0,T)\bigr)^2 \times H^k(0,T) \times M^k(Q_T)$ into the space $X^k(Q_T)$.
\end{theorem}

Under certain additional assumptions on the function $f$ the solution of the initial-boundary value problem possesses also smoothness with respect to $x$.

\begin{corollary}\label{C1.1}
Let the functions $g$, $u_0$, $\mu_0$, $\nu_0$, $\nu_1$ and $f$ satisfy the hypotheses of Theorem~\ref{T1.4} and, in addition,
$$
\partial_t^m f\in C([0,T];H^{3(k-m-1)}(I)) \cap L_2(0,T;H^{3(k-m)-2}(I),\quad 0\leq m\leq k-1.
$$
Let $u\in X^k(Q_T)$ be the solution to problem \eqref{1.1}--\eqref{1.3} for certain $T>0$, then
$$
\partial_t^m u \in C([0,T];H^{3(k-m)}(I)) \cap L_2(0,T;H^{3(k-m)+1}(I)), \quad 0\leq m\leq k-1.
$$ 
\end{corollary}

The paper is organized as follows. In Section~\ref{S2} certain auxiliary nonlinear estimates are proved. Section~\ref{S3} contains results on auxiliary initial-boundary value problem for the corresponding linear equation. In Section~\ref{S4} is devoted to inverse problems for this linear equation. In Section~\ref{S5} the main nonlinear theorems are proved.

\section{Auxiliary nonlinear estimates}\label{S2}

\begin{lemma}\label{L2.1}
There exists a constant $c=c(R)$ such that for any $u\in X^k(Q_T)$
\begin{equation}\label{2.1}
\sup\limits_{(t,x)\in Q_T} |\partial_t^m u(t,x)| \leq c \|u\|_{X^k(Q_T)},\quad 0\leq m\leq k-1.
\end{equation}
\end{lemma}

\begin{proof}
The proof is based on the elementary interpolating inequality
\begin{equation}\label{2.2}
\sup\limits_{x\in I} |\varphi(x)| \leq c(R)\bigl(\|\varphi'\|_{L_2(I)} + \|\varphi\|_{L_2(I)}\bigr)
\end{equation}
and the elementary equality
$$
\int_I v^2(t,x)\,dx = \int_I v^2(0,x)\,dx + 2\iint_{Q_t} v_\tau(\tau,x) v(\tau,x) \,dx d\tau.
$$
Then with the use of \eqref{1.6}
\begin{multline*}
\sup\limits_{(t,x)\in Q_T} v^2(t,x) \leq c\sup\limits_{t\in (0,T)} \Bigl(\int_I v_x^2(t,x)\,dx + \int_I v^2(t,x)\,dx \Bigr)  \leq
c\Bigl(\int_I v_x^2(0,x)\,dx \\+ \iint_{Q_T} v^2_{tx} (t,x)\,dxdt +  \iint_{Q_T} v^2_{x} (t,x)\,dxdt\Bigr) + c\sup\limits_{t\in (0,T)} \int_I v^2(t,x) \, dx\leq c\|v\|^2_{X^1(Q_T)},
\end{multline*}
whence \eqref{2.2} follows.
\end{proof}

\begin{lemma}\label{L2.2}
Let the function $g$ satisfy condition \eqref{1.8}.
Then there exists a nondecreasing function $\beta(r)$ such that if $\|u\|_{X^k(Q_T)}, \|u\|_{X^k(Q_T)} \leq r$ for certain $r>0$
\begin{equation}\label{2.3}
\bigl\| \partial_t^m \bigl(g(u)-g(v)\bigr) \bigr\|_{L_2(Q_T)}\leq T^{1/2}\beta(r) r \|u-v\|_{X^k(Q_T)}, \quad 0\leq m \leq k;
\end{equation}
in particular,
\begin{equation}\label{2.4}
\bigl\| \partial_t^m g(u) \bigr\|_{L_2(Q_T)} \leq T^{1/2}\beta(\|u\|_{X^k(Q_T)}) \|u\|_{X^k(Q_T)}^2, \quad 0\leq m \leq k.
\end{equation}
\end{lemma}

\begin{proof}
First of all note that  according to \eqref{2.1}
$$
\|\partial_t^l u\|_{C(\overline Q_T)}, \|\partial_t^l v\|_{C(\overline Q_T)} \leq c r, \quad 0\leq l\leq k-1.
$$
Then since $|g'(\theta)| \leq c(r)r$ if $|\theta|\leq r$
$$
\|g(u)-g(v)\|_{L_2(Q_T)} \leq c(r)r \|u-v\|_{L_2(Q_T)} \leq T^{1/2} c(r)r \|u-v\|_{X^1(Q_T)}.
$$
For $m\geq 1$ apply equality \eqref{1.9}. Then in the right-hand side of this equality the summing is in fact  performed from $l=1$. 
For $l=1$
\begin{multline*}
\|g'(u)\partial_t^m u - g'(v)\partial_t^m v\|_{L_2(Q_T)} \leq c(r) \|u-v\|_{C(\overline Q_T)} \|\partial_t^m u\|_{L_2(Q_T)} \\+c(r) \|v\|_{C(\overline Q_T)} \|\partial_t^m (u-v)\|_{L_2(Q_T)} 
\leq T^{1/2}c_1(r)r  \|u-v\|_{X^m(Q_T)}
\end{multline*}
and for $2\leq l\leq m$ 
$$
\|g^{(l)}(u)-g^{(l)}(v)\|_{L_2(Q_T)} \leq c(r)\|u-v\|_{L_2(Q_T)} \leq T^{1/2} c(r) \|u-v\|_{X^1(Q_T)},
$$
$n_j\leq m-1$ $\forall \ 1\leq j \leq l$ and so
$$
\|g^{(l)}(u) \partial_t^{n_1}u\dots \partial_t^{n_l}u - g^{(l)}(v) \partial_t^{n_1}v\dots \partial_t^{n_l}v\|_{L_2(Q_T)}  \leq
T^{1/2} c(r)r \|u-v\|_{X^m(Q_T)}.
$$
As a result, inequality \eqref{2.3} follows. Inequality \eqref{2.4} is a particular case of \eqref{2.3} for $v\equiv 0$.
\end{proof}

\begin{lemma}\label{L2.3}
Let the function $g$ satisfy condition \eqref{1.8}.
For any function $u\in X^k(Q_T)$ define
\begin{equation}\label{2.5}
\rho(u) \equiv \sum\limits_{m=0}^{k-1} \|\partial_t^m u\big|_{t=0} \|_{H^{3(k-m)}(I)}.
\end{equation}
Then there exists a nondecreasing function $\lambda(\rho)$ such that if $\rho(u), \rho(v) \leq \rho$ for certain $\rho>0$
\begin{equation}\label{2.6}
\bigl\| \partial_t^m \bigl(g(u)-g(v)\bigr)_x\big|_{t=0} \bigr\|_{H^{3(k-m-1)}(I)}\leq \lambda(\rho) \rho(u-v), \quad 0\leq m \leq k-1;
\end{equation}
in particular,
\begin{equation}\label{2.7}
\bigl\| \partial_t^m \bigl(g(u)\bigr)_x\big|_{t=0} \bigr\|_{H^{3(k-m-1)}(I)} \leq \lambda(\rho(u)) \rho(u), \quad 0\leq m \leq k-1.
\end{equation}
\end{lemma}

\begin{proof}
Note that according to \eqref{1.9} while calculating $\partial_x^j \partial_t^m g(u)$, where $1\leq j \leq 3(k-m)-2$ one obtains a linear combination of expressions of a type
$$
g^{(i)}(u) \prod \partial_x^n \partial_x^l u,
$$
 where $1\leq i\leq 3k-2$, $l\leq m$, $1\leq n \leq 3(k-l) -2$, with at least one factor $\partial_x^n \partial_x^l u$. Then since by virtue of \eqref{2.2}
$$
\|\partial_t^l u\big|_{t=0}\|_{C^{3(k-l)-2}(\overline I)},
\|\partial_t^l v\big|_{t=0}\|_{C^{3(k-l)-2}(\overline I)}
\leq c \rho(u),
$$
$$
\bigl\|\bigl(g^{(i)}(u) - g^{(i)}(v)\bigr)\big|_{t=0}\bigr\|_{L_2(I)} \leq c(\rho) \rho(u-v), \quad 
\|g^{(i)}(v)\big|_{t=0}\|_{C(\overline I)} \leq c(\rho),
$$
inequality \eqref{2.6} follows. Inequality \eqref{2.7} is a particular case of \eqref{2.6} for $v\equiv 0$.
\end{proof}

\begin{lemma}\label{L2.4}
Let the function $g$ satisfy condition \eqref{1.8}.
Then for any pair of functions $(u_0(x),f(t,x))$, satisfying $u_0\in H^{3k}(I)$, 
$f \in M^k(Q_T)$, the functions $\Phi_m(x;u_0,f)\in H^{3(k-m)}(I)$ for $0\leq m \leq k$ exist. Moreover, define
\begin{equation}\label{2.8}
\varkappa(u_0,f) = \|u_0\|_{H^{3k}(I)} + \sum\limits_{m=0}^{k-1} 
\|\partial_t^m f\big|_{t=0}\|_{H^{3(k-m-1)}(I)};
\end{equation}
then there exists a nondecreasing function $\sigma(\varkappa)$ such that if $\varkappa(u_0,f), \varkappa(\widetilde u_0,\widetilde f) \leq \varkappa$ for certain $\varkappa>0$
\begin{equation}\label{2.9}
\|\Phi_m(\cdot;u_0,f) - \Phi_m(\cdot;\widetilde u_0,\widetilde f)\|_{H^{3(k-m)}(I)} \leq
\sigma(\varkappa) \varkappa(u_0-\widetilde u_0, f-\widetilde f), \quad 
0\leq m \leq k;
\end{equation}
in particular,
\begin{equation}\label{2.10}
\|\Phi_m(\cdot;u_0,f)\|_{H^{3(k-m)}(I)} \leq \sigma(\varkappa(u_0,f)) \varkappa(u_0,f),
\quad 0\leq m \leq k.
\end{equation}
\end{lemma}

\begin{proof}
Apply induction with respect to $m$. For $m=0$ the assertion is obvious. Then for $m>0$ by \eqref{1.10} one has to calculate
$$
\bigl(g^{(l)}(u_0) \Phi_{n_1}(u_0,f)\cdots\Phi_{n_l}(u_0,f)\bigr)^{(j)}
$$
for $1\leq l \leq m-1$, $1\leq j \leq 3(k-m)+1$, $n_i \leq m-1$. Obviously, it is a linear combination of expressions of a type
$$
g^{(i)}(u_0) \prod \Phi_{n}^{(s)}(u_0,f),
$$
where $1\leq i \leq 3k-2$, $n\leq m-1$, $s\leq 3(k-m)+1$, and, thus, exists. Since here by estimate \eqref{2.10} for $m-1$
$$
\|\Phi_n(u_0,f)\|_{C^{3(k-m)+2}(\overline I)}, 
\|\Phi_n(\widetilde u_0,\widetilde f)\|_{C^{3(k-m)+2}(\overline I)} \leq c(\varkappa),
$$
$$
\|g^{(i)}(u_0) - g^{(i)}(\widetilde u_0)\|_{L_2(I)} \leq
c(\varkappa) \|u_0 - \widetilde u_0\|_{L_2(I)}, 
$$
$$
\|\Phi^{(s)}_n(u_0,f) - \Phi^{(s)}_n(\widetilde u_0,\widetilde f)\|_{L_2(I)} \leq 
c(\varkappa) \varkappa(u_0-\widetilde u_0, f-\widetilde f),
$$
we find that
\begin{multline*}
\bigl\|\bigl(g^{(l)}(u_0) \Phi_{n_1}(u_0,f)\cdots\Phi_{n_l}(u_0,f) -
g^{(l)}(\widetilde u_0) \Phi_{n_1}(\widetilde u_0,\widetilde f)\cdots
\Phi_{n_l}(\widetilde u_0,\widetilde f)\bigr)'\|_{H^{3(k-m)}(I)} \\ \leq c(\varkappa)
\varkappa(u_0- \widetilde u_0, f-\widetilde  f).
\end{multline*}
Finally, note that since $m\leq k$
$$
\|\partial_t^{m-1} f\big|_{t=0} - \partial_t^{m-1} \widetilde f\big|_{t=0}\|_{H^{3(k-m)}(I)}
\leq \varkappa(u_0 -\widetilde u_0, f-\widetilde f),
$$
and by \eqref{2.9} for $m-1$
\begin{multline*}
\bigl\|\bigl(b\Phi'_{m-1}(\cdot;u_0,f) + \Phi'''_{m-1}(\cdot;u_0,f)\bigr) -
\bigl(b\Phi'_{m-1}(\cdot;\widetilde u_0,\widetilde f) + 
\Phi'''_{m-1}(\cdot;\widetilde u_0,\widetilde f)\bigr)\bigr\|_{H^{3(k-m)}(I)} \\ \leq
c(\varkappa) \varkappa(u_0-\widetilde u_0, f-\widetilde f),
\end{multline*}
and finish the proof of existence of $\Phi_m$ and inequality \eqref{2.9}. Inequality \eqref{2.10} is a particular case of \eqref{2.9} for $\widetilde u_0\equiv 0$, $\widetilde f\equiv 0$.
\end{proof}

\section{Auxiliary linear problem}\label{S3}

Consider in a rectangle $Q_T$ an initial-boundary value problem for a linear equation
\begin{equation}\label{3.1}
u_t + bu_x + u_{xxx} = f(t,x)
\end{equation}
with initial and boundary conditions \eqref{1.2}, \eqref{1.3}. 

\begin{definition}\label{D3.1}
Let $u_0\in L_2(I)$, $\mu_0,\nu_0,\nu_1\in L_2(0,T)$, $f\equiv f_0+f_{1x}$, where $f_0\in L_1(0,T;L_2(0,R))$, $f_1\in L_2(Q_T)$.
A function $u\in L_\infty(0,T;L_2(0,R))$ is called a weak solution of problem \eqref{3.1}, \eqref{1.2}, \eqref{1.3} if for any function $\phi\in L_2(0,T;H^3(0,R))$, satisfying $\phi_t\in L_2(Q_T)$, $\phi\big|_{t=T}\equiv 0$, $\phi\big|_{x=0}= \phi_x\big|_{x=0}= \phi\big|_{x=R}\equiv 0$, the following identity is satisfied:
\begin{multline}\label{3.2}
\iint_{Q_T} \Bigl[u(\phi_t+b\phi_x+\phi_{xxx}) +f_0\phi -f_1\phi_x\Bigr]\,dxdt + 
\int_I u_0\phi\big|_{x=0}\,dx \\+ 
\int_0^T \Bigl(\mu_0\phi_{xx}\big|_{x=0} - \nu_0\phi_{xx}\big|_{x=R}+ \nu_1\phi_x\big|_{x=R}\Bigr)\,dt =0.
\end{multline}
\end{definition}

Introduce a function space
\begin{multline*}
M^k_1(Q_T) = \{ f(t,x) = f_{1x}(t,x): \partial_t^m f_1 \in L_2(Q_T),\ 0\leq m \leq k;\\ \partial_t^m f\big|_{t=0} \in H^{3(k-m-1)}(I), \ 0\leq m \leq k-1\}
\end{multline*}
endowed with the natural norm
\begin{equation}\label{3.3}
\|f\|_{M^k_1(Q_T)} = \sum\limits_{m=0}^k \|\partial_t^m f_1\|_{L_2(Q_T)} + \sum\limits_{m=0}^{k-1} \|\partial_t^m f\big|_{t=0} \|_{H^{3(k-m-1)}(I)}.
\end{equation}

\begin{remark}\label{R3.1}
It is obviously assumed that any function $f\in M_1^k(Q_T)$ is already changed on a zero measure set providing $\partial_t^m f \in C([0,T]; H^{-1}(I))$, $0\leq m \leq k-1$.
\end{remark}

In order to define compatibility conditions for the linear problem introduce the following auxiliary functions $\widetilde\Phi_m(x) = \widetilde\Phi_m(x;u_0,f)$: let $\widetilde\Phi_0 \equiv u_0$ and for $m\geq 1$
\begin{equation}\label{3.4}
\widetilde\Phi_m \equiv \partial_t^{m-1} f\big|_{t=0} - 
(b\widetilde\Phi'_{m-1} + \widetilde\Phi'''_{m-1}).
\end{equation}

Similarly to Lemma~\ref{L2.4} (of course, with more simple argument due to the absence of nonlinearity) the functions $\widetilde\Phi_m(x;u_0,f)\in H^{(3(k-m)}(I)$, $0\leq m \leq k$, exist if $u_0\in H^{3k}(I)$, 
$\partial_t^m f\big|_{t=0} \in H^{3(k-m-1)}(I)$ for $0\leq m\leq k-1$,  and
\begin{equation}\label{3.5}
\|\widetilde\Phi_m(\cdot;u_0,f)\|_{H^{3(k-m)}(I)} \leq c 
\varkappa(u_0,f), \quad 0\leq m\leq k.
\end{equation}

\begin{theorem}\label{T3.1}
Let $u_0\in H^{3k}(I)$, $\mu_0, \nu_0 \in H^{k+1/3}(0,T)$, $\nu_1\in H^k(0,T)$, $f\equiv f_0 +f_{1x}$, $f_0\in M^k(Q_T)$, $f_{1x} \in M^k_1(Q_T)$, $\mu_0^{(m)}(0) = \widetilde\Phi_{m}(0;u_0,f)$, $\nu_0^{(m)}(0) = \widetilde\Phi_m(R;u_0,f)$, $\nu_1^{(m)}(0) =  \widetilde\Phi_m'(R;u_0,h_0)$ for $0\leq m\leq k-1$. Then there exists a unique solution $u\in X^k(Q_T)$ of problem \eqref{3.1}, \eqref{1.2}, \eqref{1.3} and 
\begin{multline}\label{3.6}
\|u\|_{X(Q_{T})} \leq c(T)\Bigl[ \|u_0\|_{H^{3k}(I)} + \|\mu_0\|_{H^{k+1/3}(0,T)} + 
\|\nu_0\|_{H^{k+1/3}(0,T)}  + \|\nu_1\|_{H^k(0,T)} \\+
\|f_0\|_{M^k(Q_T)} + \|f_{1x}\|_{M^k_1(Q_T)}\Bigr],
\end{multline}
where the constant c(T) does not decrease in $T$.
Moreover, for $0< m\leq k$ the function $\partial_t^m u\in X^{k-m}(Q_T)$ is the solution to problem of \eqref{3.1}, \eqref{1.2}, \eqref{1.3} type, where $u_0,\mu_0,\nu_0,\nu_1,f$ are substituted by $\widetilde\Phi_m,\mu_0^{(m)},\nu_0^{(m)},\nu_1^{(m)},\partial_t^m f$.
\end{theorem}

\begin{proof}           
For $0\leq m\leq k$ consider a linear problem
\begin{gather}\label{3.7}
u_t + bu_x +u_{xxx} = \partial_t^m f,\\
\label{3.8}
u\big|_{t=0} =\widetilde\Phi_m,\quad u\big|_{x=0} = \mu_0^{(m)}, \quad
u\big|_{x=R} = \nu_0^{(m)}, \quad u_x\big|_{x=0} = \nu_1^{(m)}.
\end{gather}
Then (see, for example, \cite{F19-1}) there exists the unique solution of this problem $v_m \in X(Q_T)$ and
\begin{multline}\label{3.9}
\|v_m\|_{X(Q_{T})} \leq c(T)\Bigl[ \|\widetilde\Phi_m\|_{L_2(I)} + 
\|\mu_0^{(m)}\|_{H^{1/3}(0,T)} + \|\nu_0^{(m)}\|_{H^{1/3}(0,T)}  + 
\|\nu_1^{(m)}\|_{L_2(0,T)} \\+
\|\partial_t^m f_0\|_{L_{1}(0,T;L_2(I))} + \|\partial_t^m f_1\|_{L_2(Q_{T})}\Bigr].
\end{multline}
Obviously, $v_0\equiv u$. For any $0\leq m \leq k-1$ set
$$
\widetilde v_m(t,x) \equiv \widetilde\Phi_m(x) + 
\int_0^t \partial_\tau v_{m+1}(\tau,x) \,d\tau.
$$
Obviously, $\widetilde v_{mt} \equiv \widetilde v_{m+1}$. Let $\phi$ be an arbitrary test function from Definition~\ref{D3.1}, 
$$
\widetilde\phi(t,x) \equiv \int_t^T \phi(\tau,x)\,d\tau.
$$
This function, of course, also verifies the conditions on test functions. Write the corresponding identity \eqref{3.2} for the function $v_{m+1}$ and the function $\widetilde\phi$. Then after integrating by parts with the use of compatibility conditions and the definition of the function $\widetilde\Phi_{m+1}$ one derives for the function $\widetilde v_m$ and the function $\phi$ the same identity as for the function $v_m$. Then uniqueness of the solution implies that $\widetilde v_m \equiv v_m$, that is $v_{m+1} \equiv v_{mt}$ and so $\partial_t^m u \equiv v_m$ for $0\leq m\leq k$. In particular, $u\in X^k(Q_T)$.
\end{proof}

Introduce certain additional notation. Let 
$$
u = S(u_0, \mu_0, \nu_0, \nu_1, f_0, f_1)
$$
be the solution of problem \eqref{3.1}, \eqref{1.2}, \eqref{1.3} from the space $X^k(Q_T)$ under the hypothesis of Theorem~\ref{T3.1}. 

Let a function $\omega \in C(\overline{I})$. On the space of functions $u(t,x)$, lying in $L_1(I)$ for all $t\in [0,T]$, define a linear operator $Q(\omega)$ by a formula $(Q(\omega)u)(t) = q(t;u,\omega)$, where
\begin{equation}\label{3.10}
q(t;u,\omega)\equiv \int_I u(t,x) \omega(x) \,dx,\quad t\in [0,T].
\end{equation}

\begin{lemma}\label{L3.1}
Let the hypothesis of Theorem~\ref{T3.1} be satisfied. Let the function $\omega$ satisfy conditions \eqref{1.11}. Then for the function $u =  S(u_0, \mu_0, \nu_0, \nu_1, f_0, f_1)$ the corresponding function $q(\cdot;u,\omega) = Q(\omega)u \in H^{k+1}(0,T)$ and for $0\leq m \leq k$
\begin{multline}\label{3.11}
q^{(m+1)}(t;u,\omega) = r(t;u,\omega,m) \equiv \nu_1^{(m)}(t) \omega'(R) + 
\mu_0^{(m)}(t) \omega''(0) - \nu_0^{(m)}(t) \omega''(R) \\ +
\int_I \bigl[\partial_t^m u(t,x) \bigl(b \omega'(x) + \omega'''(x)\bigr) +
\partial_t^m f_0(t,x) \omega(x) - \partial_t^m f_1(t,x)\omega'(x)\bigr]\, dx;
\end{multline}
moreover,
\begin{multline}\label{3.12}
\|q^{(m+1)}(\cdot;u,\omega)\|_{L_2(0,T)} \leq c(T) \Bigl[ T^{1/2}\|u\|_{X^m(Q_T)} +\|\mu_0^{(m)}\|_{L_2(0,T)} \\+ \|\nu_0^{(m)}\|_{L_2(0,T)} + \|\nu_1^{(m)}\|_{L_2(0,T)} + \|\partial_t^m f_0\|_{L_{2}(Q_T)} + \|\partial_t^m f_1\|_{L_{2}(Q_T)}\Bigr],
\end{multline}
where the constant $c(T)$ does not decrease in $T$.
\end{lemma}

\begin{proof}
According to Theorem~\ref{T3.1} write for the function $\partial_t^m u$ the corresponding identity \eqref{3.2}:
\begin{multline}\label{3.13}
\iint_{Q_T} \Bigl[\partial_t^m u(\phi_t+b\phi_x+\phi_{xxx}) +\partial_t^m f_0\phi -\partial_t^m f_1\phi_x\Bigr]\,dxdt + 
\int_I \widetilde\Phi_m\phi\big|_{x=0}\,dx \\+ 
\int_0^T \Bigl(\mu_0^{(m)}\phi_{xx}\big|_{x=0} - \nu_0^{(m)}\phi_{xx}\big|_{x=R}+ \nu_1^{(m)}\phi_x\big|_{x=R}\Bigr)\,dt =0.
\end{multline}
For an arbitrary function $\psi \in C_0^\infty(0,T)$ let $\phi(t,x) \equiv \psi(t)\omega(x)$. This function satisfies the assumption on a test function from Definition~\ref{D3.1} and then identity  \eqref{3.13} after integration by parts yields that
\begin{equation}\label{3.14}
\int_0^T \psi^{(m+1)}(t) q(t;u,\omega) \, dt = (-1)^{m+1} \int_0^T \psi(t) r(t;u,\omega,m) \,dt.
\end{equation}
Since $r(\cdot;u,\omega,m)\in L_2(0,T)$, it follows from \eqref{3.14} that there exists the Sobolev derivative $q^{(m+1)}(t;u,\omega) = r(t;u,\omega,m) \in L_2(0,T)$.
Since $\|\partial_t^m u\|_{L_2(0,T;L_2(I))} \leq T^{1/2} 
\|\partial_t^m u\|_{C([0,T];L_2(I))}$, inequality \eqref{3.12} follows from \eqref{3.11}.
\end{proof}

\section{Linear inverse problems}\label{S4}

The following lemma is crucial for the first inverse problem in the linear case.

\begin{lemma}\label{L4.1}
Let the functions $\omega_j$, $j=1,\, 2$, satisfy conditions \eqref{1.11}, 
$h\in C^k ([0,T];L_2(I))$ and condition \eqref{1.14} be satisfied. Let the functions $\varphi_j \in \widetilde 	H^{k+1}(0,T)$. Then there exists a unique pair of functions $(F,\nu) \equiv \Gamma_1(\varphi_1,\varphi_2)\in \bigl(\widetilde H^k(0,T)\bigr)^2$ such that the function $u = S(0,0,0,\nu,F h,0) \in X^k(Q_T)$ verifies conditions \eqref{1.4}. Moreover, the linear operator $\Gamma_1: \bigl(\widetilde H^{k+1}(0,T)\bigr)^2 \to \bigl(\widetilde H^k(0,T)\bigr)^2$ is bounded and its norm does not decrease in $T$.
\end{lemma}

\begin{proof}
First of all note, that if $(F,\nu) \in \bigl(\widetilde H^k(0,T)\bigr)^2$ then, in particular, $F h \in M^k(Q_T)$, $\widetilde\Phi_m(\cdot;0,F h) \equiv 0$ for $0\leq m\leq k$, so the hypothesis of Theorem~\ref{T3.1} for problem
\begin{gather}\label{4.1}
u_t +b u_x +u_{xxx} = F h,\\
\label{4.2}
u\big|_{t=0}=0,\quad u\big|_{x=0} = u\big|_{x=R} =0,\quad u_x\big|_{x=R} =\nu
\end{gather}
is satisfied and the function $u= S(0,0,0,\nu,F h,0) \in X^k(Q_T)$ exists as the solution of this problem.

On the space $\bigl(\widetilde H^{k}(0,T)\bigr)^2$ introduce two linear operators $\Lambda_j(F,\nu) = Q(\omega_j)S(0,0,0,\nu,F h,0)$, $j=1,\, 2$. Let $\Lambda = (\Lambda_1, \lambda_2)$. Then   Lemma~\ref{L3.1} provides that the operator $\Lambda$ acts from the space $\bigl(\widetilde H^{k}(0,T)\bigr)^2$ into the space $\bigl(\widetilde H^{k+1}(0,T)\bigr)^2$ and is bounded.

Note that the equality $(\varphi_1, \varphi_2) = \Lambda (F,\nu)$ for $F, \nu \in \widetilde H^k(0,T)$ obviously means that the functions $F, \nu$ are the desired ones. 

Define for the function $u=  S(0,0,0,\nu,F h,0)$
\begin{equation}\label{4.3}
\widetilde r_j(t) \equiv \int_I u(t,x) \bigl(b \omega'_j(x) + \omega'''_j(x)\bigr)\, dx
\in \widetilde H^k(0,T).
\end{equation}
Then it follows from \eqref{3.11} for $m=0$ (here $f_0 \equiv F h$, $f_1\equiv 0$, $\mu_0=\nu_0 \equiv 0$, $\nu_1\equiv \nu$) that for $q_j(t) \equiv q(t;u,\omega_j) = 
\bigl(\Lambda_j (F,\nu)\bigr)(t)$
\begin{equation}\label{4.4}
q_j'(t) = r(t;u,\omega_j,0) =\widetilde r_j(t;u) + \nu(t) \omega'_j(R)  +  F(t) \psi_j(t) \in
\widetilde H^k(0,T).
\end{equation}
Therefore, if
\begin{equation}\label{4.5}
y_j(t) \equiv q'_j(t) - \widetilde r_j(t; u), \quad
\widetilde \Delta_1(t) \equiv 
\begin{vmatrix}
y_1(t) & \omega_1'(R)\\
y_2(t) & \omega_2'(R)
\end{vmatrix}, \quad
\widetilde \Delta_2(t) \equiv 
\begin{vmatrix}
\psi_1(t) & y_1(t) \\\psi_2(t) & y_2(t)
\end{vmatrix},
\end{equation}
it follows from \eqref{4.4}, \eqref{4.5} that
\begin{equation}\label{4.6}
(F,\nu)(t) = \Bigl(\frac{\widetilde \Delta_1(t)}{\Delta(t)}, \frac{\widetilde \Delta_2(t)}{\Delta(t)} \Bigr).
\end{equation}
Set for $j=1,\, 2$
\begin{equation}\label{4.7}
z_j(t) \equiv \varphi'_j(t) - \widetilde r_j(t; u) \in \widetilde H^k(0,T)
\end{equation}
and let
\begin{equation}\label{4.8}
\Delta_1(t) \equiv 
\begin{vmatrix}
z_1(t) & \omega_1'(R)\\
z_2(t) & \omega_2'(R)
\end{vmatrix}, \quad
\Delta_2(t) \equiv 
\begin{vmatrix}
\psi_1(t) & z_1(t) \\
\psi_2(t) & z_2(t)
\end{vmatrix},\quad \Delta_j \in \widetilde H^k(0,T).
\end{equation}
Introduce operators $A_j: \bigl(\widetilde H^k(0,T)\bigr)^2 \to 
\widetilde H^k(0,T)$ in the following way:
\begin{equation}\label{4.9}
A_j(F, h)(t) \equiv \frac{\Delta_j(t)}{\Delta(t)}.
\end{equation}
Let $A(F,\nu) = \bigl(A_1(F,\nu), A_2(F,\nu)\bigr) \in \bigl(\widetilde H^k(0,T)\bigr)^2$. 
Note that $(\varphi_1, \varphi_2) = \Lambda(F, \nu)$ if and only if $A(F, \nu)=(F, \nu)$.

Indeed, if $(\varphi_1, \varphi_2) = \Lambda(F,\nu)$ then $\varphi_j'(t) \equiv q_j'(t)$, $z_j(t) \equiv y_j(t)$ and equalities \eqref{4.5}, \eqref{4.8} yield that $\Delta_j(t) \equiv \widetilde\Delta_j(t)$ and so $A(F,\nu) = (F,\nu)$ because of \eqref{4.6}, \eqref{4.9}.

Vice versa, if $A(F,\nu)=(F, \nu)$, then $\Delta_j(t) \equiv \widetilde \Delta_j(t)$ and condition $\Delta(t)\ne 0$ implies that $\varphi_j'(t) \equiv q_j'(t)$. Since $q_j(0) = \varphi_j(0) =0$, the equality $\varphi_j(t) \equiv q_j(t)$ follows.

Next, we show that the operator $A$ is a contraction under the choice of a special norm in the space $\bigl(\widetilde H^k(0,T)\bigr)^2$. Note that $|\Delta(t)| \geq \Delta_0 >0$ for all $t\in [0,T]$.

Let $F_i, \nu_i \in \widetilde H^k(0,T)$, $i =1,\, 2$, $u_i \equiv S(0,0,0,\nu_i,F_i h,0)$, then according to \eqref{4.7}--\eqref{4.9}
\begin{multline}\label{4.10}
\bigl(A(F_1, \nu_1) - A(F_2, \nu_2)\bigr)(t) \\ = - \frac1{\Delta(t)}
\Bigl(
\begin{vmatrix}
\widetilde r_1(t;u_1-u_2) & \omega_1'(R)\\
\widetilde r_2(t;u_1-u_2) & \omega_2'(R)
\end{vmatrix}, 
\begin{vmatrix}
\psi_1(t) & \widetilde r_1(t;u_1-u_2) \\
\psi_2(t) & \widetilde r_2(t;u_1-u_2),
\end{vmatrix}  \Bigr),
\end{multline}
and since $\omega_j\in H^3(I)$, $\psi_j\in C^k[0,T]$, it follows from \eqref{4.3} and \eqref{3.6} that for $t\in [0,T]$
\begin{multline}\label{4.11}
\sum\limits_{m=0}^k \bigl|\partial_t^m\bigl(A(F_1, \nu_1) - A(F_2, \nu_2)\bigr)\bigr|(t) \leq
c \sum\limits_{m=0}^k \bigl\|\partial_t^m\bigl(u_1(t,x) - u_2(t,x)\bigr)\bigr\|_{L_2(I)} \\ \leq
c(T) \sum\limits_{m=0}^k \Bigl[ \|F_1^{(m)}(\tau) - F_2^{(m)}(\tau)\|_{L_2(0,t)} +
\|\nu_1^{(m)}(\tau) - \nu_2^{(m)}(\tau)\|_{L_2(0,t)}\Bigr].
\end{multline}
For $\gamma>0$ let
\begin{equation}\label{4.12} 
\|\varphi\|_{H_\gamma^k(0,T)} = \sum\limits_{m=0}^k \|e^{-\gamma t} \varphi^{(m)}(t)\|_{L_2(0,T)},
\end{equation}
then inequality \eqref{4.11} implies that 
\begin{multline}\label{4.13}
\|A(F_1,\nu_1) - A(F_2,\nu_2)\|_{(H^k_{\gamma}(0,T))^2} \\  \leq
c \sum\limits_{m=0}^k \Bigl[\int_0^T e^{-2\gamma t} \int_0^t 
\Bigl(\bigl(F_1^{(m)}(\tau) - F_2^{(m)}(\tau)\bigr)^2 + 
\bigl(\nu_1^{(m)}(\tau) - \nu_2^{(m)}(\tau)\bigr)^2\Bigr)\,d\tau\,dt\Bigr]^{1/2} \\ =
c \sum\limits_{m=0}^k \Bigl[\int_0^T e^{-2\gamma \tau} \Bigl(
\bigl(F_1^{(m)}(\tau) - F_2^{(m)}(\tau)\bigr)^2 + \bigl(\nu_1^{(m)}(\tau) - \nu_2^{(m)}(\tau)\bigr)^2 \Bigr)
\\ \times \int_\tau^T e^{2\gamma(\tau - t)}\,dt \,d\tau \Bigr]^{1/2} 
\leq \frac{c}{\gamma^{1/2}} \|(F_1,\nu_1) - (F_2,\nu_2)\|_{(H_{\gamma}^k(0,T))^2}.
\end{multline}
It remains to choose respectively large $\gamma$.

As a result, for any pair of functions $(\varphi_1, \varphi_2) \in \bigl(\widetilde H^{k+1}(0,T)\bigr)^2$ there exists a unique pair of functions $(F,\nu) \in \bigl(\widetilde H^k(0,T)\bigr)^2$ verifying $A(F,\nu)=(F, \nu)$, that is $\varphi_j = \Lambda_j (F,\nu)$. It means that the operator $\Lambda$ is invertible and so the Banach theorem implies that the inverse operator $\Gamma_1= \Lambda^{-1} : \bigl(\widetilde H^{k+1}(0,T)\bigr)^2 \to \bigl(\widetilde H^k(0,T)\bigr)^2$ is continuous. In particular,
\begin{equation}\label{4.14}
\|\Gamma_1(\varphi_1,\varphi_2)\|_{(\widetilde H^k(0,T))^2} \leq c(T) 
\|(\varphi_1,\varphi_2)\|_{(\widetilde H^{k+1}(0,T))^2}.
\end{equation}

For an arbitrary $T_1>T$ extend the functions $\varphi_j$ by the Taylor polynomials of degree $k$ with the center $T$ to the interval $(T,T_1)$, Then the extended functions $\varphi_j \in \widetilde H^{k+1}(0,T_1)$ and $\|\varphi_j^{(k+1)}\|_{L_2(0,T_1)} = 
\|\varphi_j^{(k+1)}\|_{L_2(0,T)}$. Then the analog of inequality \eqref{4.14} on the interval $(0,T_1)$ for such a function evidently holds with $c(T_1) \geq c(T)$. It means that the norm of the operator $\Gamma$ in non-decreasing in $T$.
\end{proof}

Lemma~\ref{L4.1} provides the following results on the first linear inverse problem. 

\begin{theorem}\label{T4.1}
Let 
\begin{itemize}
\item
$u_0 \in H^{3k}(I)$, $\mu_0, \nu_0 \in H^{k+1/3}(0,T)$, $h_0 \in M^k(Q_T)$, $f_{1x} \in M^k_1(Q_T)$ and
$\mu_0^{(m)}(0) = \widetilde\Phi_{m}(0;u_0,h_0+f_{1x})$, $\nu_0^{(m)}(0) = \widetilde\Phi_m(R;u_0,h_0+f_{1x})$ for $0\leq m\leq k-1$; 
\item
for $j=1, \, 2$ the functions $\omega_j$ satisfy condition \eqref{1.11}, $\varphi_j \in H^{k+1}(0,T)$  and
\begin{equation}\label{4.15}
\varphi_j^{(m)}(0) = \int_I \widetilde\Phi_m(x;u_0,h_0+f_{1x}) \omega_j(x) \,dx, \quad 0\leq m \leq k;
\end{equation}
\item
$h\in C^k([0,T]; L_2(I))$ and condition \eqref{1.14} be satisfied.
\end{itemize}
Then there exist a unique pair of functions $(F, \nu_1)$, satisfying
$F\in \widetilde H^k(0,T)$, $\nu_1 \in H^k(0,T)$, $\nu_1^{(m)}(0) = \widetilde\Phi_m'(R;u_0,h_0+f_{1x})$ for $0\leq m \leq k-1$, and the corresponding unique solution of problem \eqref{3.1}, \eqref{1.2}, \eqref{1.3} $u\in X^k(Q_T)$ satisfying conditions \eqref{1.4} for $j=1, \, 2$, where 
\begin{equation}\label{4.16}
f(t,x) \equiv h_0(t,x) + F(t)h(t,x) +f_{1x}(t,x).
\end{equation}
Moreover, if $\nu^* \in H^k(0,T)$ is an arbitrary function satisfying 
$(\nu^*)^{(m)}(0) = \widetilde\Phi'_m(R;u_0,h_0+f_{1x})$  for $0\leq m\leq k-1$ and if
\begin{equation}\label{4.17}
w \equiv S(u_0,\mu_0,\nu_0,\nu^*,h_0,f_1),
\end{equation}
then the functions $F$, $\nu_1$ and $u$ are given by formulas
\begin{gather}\label{4.18}
(F,\nu) = \Gamma_1\bigl(\varphi_1- Q(\omega_1)w, \varphi_2- Q(\omega_2)w \bigr),\quad
\nu_1 = \nu + \nu^*, \\
\label{4.19}
u = S(u_0,\mu_0,\nu_0,\nu_1,h_0+F h,f_1),
\end{gather}
where the operator $\Gamma_1$ is defined in Lemma~\ref{L4.1}.
\end{theorem}

\begin{proof}
The function $w$ is the solution of the problem
\begin{gather}\label{4.20}
w_t +bw_x +w_{xxx} = h_0 + f_{1x},\\
\label{4.21}
w\big|_{t=0} = u_0, \quad w\big|_{x=0} = \mu_0,\quad w\big|_{x=R}= \nu_0,\quad
w_x\big|_{x=R} = \widetilde\nu.
\end{gather}
The hypothesis of Theorem~\ref{T3.1} for this problem is satisfied and, therefore, the function
$w=  S(u_0,\mu_0,\nu_0,\nu^*,h_0,f_1) \in X^k(Q_T)$ exists. Moreover, 
since $\partial_t^m w(0,x) = \widetilde\Phi_m(0;u_0,h_0+f_{1x})$ for $0\leq m\leq k$, conditions \eqref{4.15} ensure that $\varphi_j^{(m)}(0) = q^{(m)}(0;w,\omega_j)$ and, therefore $\widetilde\varphi_j \equiv \varphi_j - Q(\omega_j)w \in \widetilde H^k(0,T)$.
As a result, $\Gamma_1(\widetilde\varphi_1, \widetilde\varphi_2) \in 
\bigl(\widetilde H^k(0,T)\bigr)^2$ is defined and so formulas \eqref{4.18}, \eqref{4.19} define the functions $F, \nu_1, u$ from the corresponding spaces. 

It follows, from \eqref{4.1}, \eqref{4.2}, \eqref{4.20}, \eqref{4.21} that the function $u= w + S(0,0,0,\nu,F h,0)$ is the solution of problem \eqref{3.1}, \eqref{1.2}, \eqref{1.3} for the function $f$ given by formula \eqref{4.16}. Moreover, $Q(\omega_j)u = Q(\omega_j)w + \widetilde\varphi_j = \varphi_j$, that is conditions \eqref{1.4} are satisfied.

Uniqueness of the functions $F$, $\nu_1$ and $u$ evidently succeeds from Lemma~\ref{L4.1} and Theorem~\ref{T3.1}.
\end{proof}

Now establish similar results for other inverse linear problems.

\begin{lemma}\label{L4.2}
Let the function $\omega_0$ satisfy condition \eqref{1.11}, 
$h\in C^k ([0,T];L_2(I))$ and condition \eqref{1.18} be satisfied. Let the function $\varphi_0 \in \widetilde H^{k+1}(0,T)$. Then there exists a unique function $F = \Gamma_2 \varphi_0 \in \widetilde H^k(0,T)$ such that the $u = S(0,0,0,0,F h,0) \in X^k(Q_T)$ verifies condition \eqref{1.4} for $j=0$. Moreover, the linear operator $\Gamma_2: \widetilde H^{k+1}(0,T) \to \widetilde H^k(0,T)$ is bounded and its norm does not decrease in $T$.
\end{lemma}

\begin{proof}
As in the proof of the previous lemma note, that if $F \in \widetilde H^k(0,T)$ then, in particular, $F h \in M^k(Q_T)$, $\widetilde\Phi_m(\cdot;0,F h) \equiv 0$ for $0\leq m\leq k$, so the hypothesis of Theorem~\ref{T3.1} for problem
\begin{gather}\label{4.22}
u_t +b u_x +u_{xxx} = F h,\\
\label{4.23}
u\big|_{t=0}=0,\quad u\big|_{x=0} = u\big|_{x=R} = u_x\big|_{x=R} =0
\end{gather}
is satisfied and the function $u= S(0,0,0,0,F h,0) \in X^k(Q_T)$ exists as the solution of this problem.

On the space $\widetilde H^{k}(0,T)$ introduce a linear operator $\Lambda F = Q(\omega_0)S(0,0,0,0,F h,0)$. Then Lemma~\ref{L3.1} provides that the operator $\Lambda$ acts from the space $\widetilde H^{k}(0,T)$ into the space $\widetilde H^{k+1}(0,T)$ and is bounded.

Note that the equality $\varphi_0 = \Lambda F$ for $F\in \widetilde H^k(0,T)$ obviously means that the function $F$ is the desired ones. 

Define for the function $u=  S(0,0,0,0,F h,0)$ the function $\widetilde r_0(t)$ by formula \eqref{4.3} for $j=0$.

Then it follows from \eqref{3.11} for $m=0$ (here $f_0 \equiv F h$, $f_1\equiv 0$, $\mu_0=\nu_0= \nu_1 \equiv 0$) that for $q_0(t) \equiv q(t;u,\omega_0) = 
(\Lambda F)(t)$
\begin{equation}\label{4.24}
q_0'(t) = r(t;u,\omega_0,0) =\widetilde r_0(t;u) +  F(t) \psi_0(t) \in
\widetilde H^k(0,T),
\end{equation}
therefore, 
\begin{equation}\label{4.25}
F(t) = \frac1{\psi_0(t)} \bigl(q_0'(t) - \widetilde r_0(t,u)\bigr).
\end{equation}
Introduce an operator $A: \widetilde H^k(0,T) \to \widetilde H^k(0,T)$ in the following way:
\begin{equation}\label{4.26}
AF(t) \equiv \frac1{\psi_0(t)} \bigl(\varphi_0'(t) - \widetilde r_0(t,u)\bigr).
\end{equation}
Then $\varphi_0 = \Lambda F$ if and only if $AF=F$.

Indeed, if $\varphi_0 = \Lambda F$ then $\varphi_0'(t) \equiv q_0'(t)$ and by \eqref{4.25}, \eqref{4.26} $AF = F$.

Vice versa, if $AF=F$, then $\varphi_0'(t) \equiv q_0'(t)$. Since $q_0(0) = \varphi_0(0) =0$ the equality $\varphi_0(t) \equiv q_0(t)$ follows.

Next, we show that the operator $A$ is a contraction under the choice of the same special norm \eqref{4.12} in the space $H^k(0,T)$. Note that $|\psi_0(t)| \geq \psi_0 >0$ for all $t\in [0,T]$.

Let $F_i \in \widetilde H^k(0,T)$, $i=1,\, 2$, $u_i \equiv S(0,0,0,0,F_i h,0)$, then according to \eqref{4.26}
\begin{equation}\label{4.27}
(A F_1 - A F_2)(t) = - \frac1{\psi_0(t)} \widetilde r_0(t;u_1-u_2)
\end{equation}
and similarly to \eqref{4.11} for $t\in [0,T]$
\begin{equation}\label{4.28}
\sum\limits_{m=0}^k \bigl|\partial_t^m (A F_1 - A F_2)\bigr|(t) \leq
c(T) \sum\limits_{m=0}^k \|F_1^{(m)}(\tau) - F_2^{(m)}(\tau)\|_{L_2(0,t)}.
\end{equation}
Then for $\gamma>0$ using formula \eqref{4.12} 
we derive that similarly to \eqref{4.13}
\begin{equation}\label{4.29}
\|A F_1 - A F_2\|_{H^k_{\gamma}(0,T)}   \leq
\frac{c}{\gamma^{1/2}} \|F_1 - F_2\|_{H^k_{\gamma}(0,T)}.
\end{equation}
As in the proof of Lemma~\ref{L4.1} it remains to choose respectively large $\gamma$.

As a result, for any function $\varphi_0 \in \widetilde H^{k+1}(0,T)$ there exists a unique  function $F \in \widetilde H^k(0,T)$ verifying $A F= F$, that is $\varphi_0 = \Lambda F$. It means that the operator $\Lambda$ is invertible and so the Banach theorem implies that the inverse operator $\Gamma_2= \Lambda^{-1} : \widetilde H^{k+1}(0,T) \to \widetilde H^k(0,T)(0,T)$ is continuous. In particular,
\begin{equation}\label{4.30}
\|\Gamma_2 \varphi_0\|_{\widetilde H^k(0,T)} \leq c(T) 
\|\varphi_0\|_{\widetilde H^{k+1}(0,T)}.
\end{equation}

Finally, as in the proof of Lemma~\ref{L4.1} one can show that the constant $c(T)$ in \eqref{4.30} does not decrease in $T$.
\end{proof}

\begin{theorem}\label{T4.2}
Let 
\begin{itemize}
\item
$u_0 \in H^{3k}(I)$, $\mu_0, \nu_0 \in H^{k+1/3}(0,T)$, $\nu_1 \in H^k(0,T)$, $h_0 \in M^k(Q_T)$, $f_{1x} \in M^k_1(Q_T)$ and
$\mu_0^{(m)}(0) = \widetilde\Phi_{m}(0;u_0,h_0+f_{1x})$, $\nu_0^{(m)}(0) = \widetilde\Phi_m(R;u_0,h_0+f_{1x})$, $\nu_1^{(m)}(0) = \widetilde\Phi'_m(R;u_0,h_0+f_{1x})$ for $0\leq m\leq k-1$; 
\item
the function $\omega_0$ satisfy condition \eqref{1.11}, $\varphi_0 \in H^{k+1}(0,T)$  and
condition \eqref{4.15} for $j=0$ be satisfied;
\item
$h\in C^k([0,T]; L_2(I))$ and condition \eqref{1.18} be satisfied.
\end{itemize}
Then there exist a unique function $F \in \widetilde H^k(0,T)$ and the corresponding unique solution of problem \eqref{3.1}, \eqref{1.2}, \eqref{1.3} $u\in X^k(Q_T)$, satisfying condition \eqref{1.4} for $j=0$, where the function $f$ is given by formula \eqref{4.16}.
Moreover, if 
\begin{equation}\label{4.31}
w \equiv S(u_0,\mu_0,\nu_0,\nu_1,h_0,f_1),
\end{equation}
then the functions $F$ and $u$ are given by formulas
\begin{equation}\label{4.32}
F = \Gamma_2 \bigl(\varphi_0- Q(\omega_0)w \bigr), \quad
u = S(u_0,\mu_0,\nu_0,\nu_1,h_0+ F h,f_1), 
\end{equation}
where the operator $\Gamma_2$ is defined in Lemma~\ref{L4.2}.
\end{theorem}

\begin{proof}
The function $w$ is the solution of the problem
\begin{gather}\label{4.33}
w_t +bw_x +w_{xxx} = h_0 + f_{1x},\\
\label{4.34}
w\big|_{t=0} = u_0, \quad w\big|_{x=0} = \mu_0,\quad w\big|_{x=R}= \nu_0,\quad
w_x\big|_{x=R} = \nu_1.
\end{gather}
The hypothesis of Theorem~\ref{T3.1} for this problem is satisfied and, therefore, the function
$w= S(u_0,\mu_0,\nu_0,\nu_1,h_0,f_1) \in X^k(Q_T)$ exists. Moreover, 
since $\partial_t^m w(0,x) = \widetilde\Phi_m(0;u_0,h_0+f_{1x})$ for $0\leq m\leq k$, conditions \eqref{4.15} ensure that $\varphi_0^{(m)}(0) = q^{(m)}(0;w,\omega_0)$ and, therefore, $\widetilde\varphi_0 \equiv \varphi_0 - Q(\omega_0)w \in \widetilde H^k(0,T)$.
As a result, $\Gamma_2 \widetilde\varphi_0 \in 
\widetilde H^k(0,T)$ is defined and so formulas \eqref{4.32} define the functions $F, u$ from the corresponding spaces. 

The end of the proof is similar to Theorem~\ref{T4.1}.
\end{proof}

\begin{lemma}\label{L4.3}
Let the function $\omega_0$ satisfy condition \eqref{1.11} and condition \eqref{1.21} be satisfied. Let the function $\varphi_0 \in \widetilde H^{k+1}(0,T)$. Then there exists a unique function $\nu = \Gamma_3 \varphi_0 \in \widetilde H^k(0,T)$ such that the corresponding function $u = S(0,0,0,\nu,0,0) \in X^k(Q_T)$ verifies condition \eqref{1.4}  for $j=0$. Moreover, the linear operator $\Gamma_3: \widetilde H^{k+1}(0,T) \to \widetilde H^k(0,T)$ is bounded and its norm does not decrease in $T$.
\end{lemma}

\begin{proof}
If $\nu \in \widetilde H^k(0,T)$ then $\widetilde\Phi(\cdot;0,0) \equiv 0$ for $0\leq m\leq k$, so the hypothesis of Theorem~\ref{T3.1} for problem
\begin{equation}\label{4.35}
u_t +b u_x +u_{xxx} = 0
\end{equation}
with initial and boundary conditions \eqref{4.2} is satisfied and the function $u= S(0,0,0,\nu,0,0) \in X^k(Q_T)$ exists as the solution of this problem.

On the space $\widetilde H^{k}(0,T)$ introduce a linear operator $\Lambda \nu = Q(\omega_0)S(0,0,0,\nu,0,0)$. Then Lemma~\ref{L3.1} provides that the operator $\Lambda$ acts from the space $\widetilde H^{k}(0,T)$ into the space $\widetilde H^{k+1}(0,T)$ and is bounded.

Note that the equality $\varphi_0 = \Lambda \nu$ for $\nu\in \widetilde H^k(0,T)$ obviously means that the function $\nu$ is the desired ones. 

Define for the function $u= S(0,0,0,\nu,0,0)$ the function $\widetilde r_0(t)$ by formula \eqref{4.3} for $j=0$.

Then it follows from \eqref{3.11} for $m=0$ (here $f_0 \equiv f_1\equiv 0$, $\mu_0=\nu_0  \equiv 0$, $\nu_1\equiv \nu$) that for $q_0(t) \equiv q(t;u,\omega_0) = 
(\Lambda F)(t)$
\begin{equation}\label{4.36}
q_0'(t) = r(t;u,\omega_0,0) =\widetilde r_0(t;u) + \nu(t) \omega'_0(R) \in
\widetilde H^k(0,T),
\end{equation}
therefore, 
\begin{equation}\label{4.37}
\nu(t) = \frac1{\omega'_0(R)} \bigl(q_0'(t) - \widetilde r_0(t,u)\bigr).
\end{equation}
Introduce operator $A: \widetilde H^k(0,T) \to \widetilde H^k(0,T)$ in the following way:
\begin{equation}\label{4.38}
A\nu(t) \equiv \frac1{\omega'_0(R)} \bigl(\varphi_0'(t) - \widetilde r_0(t,u)\bigr).
\end{equation}
Then similarly to the proof of Lemma~\ref{L4.2} $\varphi_0 = \Lambda\nu$ if and only if $A\nu=\nu$.

Next, as in the previous lemma we show that the operator $A$ is a contraction under the choice of norm \eqref{4.12} in the space $H^k(0,T)$.

Let $\nu_i \in \widetilde H^k(0,T)$, $i=1,\, 2$, $u_i \equiv  S(0,0,0,\nu_i,0,0)$, then according to \eqref{4.38}
\begin{equation}\label{4.39}
(A \nu_1 - A \nu_2)(t) = - \frac1{\omega'_0(R)} \widetilde r_0(t;u_1-u_2)
\end{equation}
and similarly to \eqref{4.11} for $t\in [0,T]$
\begin{equation}\label{4.40}
\sum\limits_{m=0}^k \bigl|\partial_t^m (A \nu_1 - A \nu_2)\bigr|(t) \leq
c(T) \sum\limits_{m=0}^k \|\nu_1^{(m)}(\tau) - \nu_2^{(m)}(\tau)\|_{L_2(0,t)}.
\end{equation}
Then for $\gamma>0$  
we derive that similarly to \eqref{4.13}
\begin{equation}\label{4.41}
\|A \nu_1 - A \nu_2\|_{H^k_{\gamma}(0,T)}   \leq
\frac{c}{\gamma^{1/2}} \|\nu_1 - \nu_2\|_{H^k_{\gamma}(0,T)}.
\end{equation}
As in the proof of Lemma~\ref{L4.1} it remains to choose respectively large $\gamma$.

The end of the proof is similar to Lemma~\ref{L4.2}, in particular, the inverse operator $\Gamma_3= \Lambda^{-1} : \widetilde H^{k+1}(0,T) \to \widetilde H^k(0,T)(0,T)$ exists, is continuous, estimate \eqref{4.30} (where $\Gamma_2$ is substituted by $\Gamma_3$) holds with the constant $c(T)$ non-decreasing in $T$.
\end{proof}

\begin{theorem}\label{T4.3}
Let 
\begin{itemize}
\item
$u_0 \in H^{3k}(I)$, $\mu_0, \nu_0 \in H^{k+1/3}(0,T)$, $f_0 \in M^k(Q_T)$, $f_{1x} \in M^k_1(Q_T)$  and
$\mu_0^{(m)}(0) = \widetilde\Phi_{m}(0;u_0,f_0+f_{1x})$, $\nu_0^{(m)}(0) = \widetilde\Phi_m(R;u_0,f_0+f_{1x})$ for $0\leq m\leq k-1$; 
\item
the function $\omega_0$ satisfy condition \eqref{1.11}, $\varphi_0 \in H^{k+1}(0,T)$  and
\begin{equation}\label{4.42}
\varphi_0^{(m)}(0) = \int_I \widetilde\Phi_m(x;u_0,f_0+f_{1x}) \omega_0(x) \,dx, \quad 0\leq m \leq k;
\end{equation}
\item
condition \eqref{1.21} be satisfied.
\end{itemize}
Then there exist a unique function $\nu_1 \in H^k(0,T)$, satisfying $\nu_1^{(m)}(0) = \widetilde\Phi'_m(R;u_0,f_0+f_{1x})$ for $0\leq m\leq k-1$, and the corresponding unique solution of problem \eqref{3.1}, \eqref{1.2}, \eqref{1.3} $u\in X^k(Q_T)$ satisfying condition \eqref{1.4} for $j=0$, where the function $f \equiv f_0 + f_{1x}$.
Moreover, if $\nu^* \in H^k(0,T)$ is an arbitrary function satisfying 
$(\nu^*)^{(m)}(0) = \widetilde\Phi'_m(R;u_0,f_0+f_{1x})$ for $0\leq m \leq k-1$ and if
\begin{equation}\label{4.43}
w \equiv S(u_0,\mu_0,\nu_0,\nu^*,f_0,f_1),
\end{equation}
then the functions $\nu_1$ and $u$ are given by formulas
\begin{equation}\label{4.44}
\nu = \Gamma_3 (\varphi_0- Q(\omega_0)w \bigr), \quad \nu_1 = \nu +\nu^*,\quad
u = S(u_0,\mu_0,\nu_0,\nu_1,f_0,f_1), 
\end{equation}
where the operator $\Gamma_3$ is defined in Lemma~\ref{L4.3}.
\end{theorem}

\begin{proof}
The function $w$ is the solution of the problem
\begin{equation}\label{4.45}
w_t +bw_x +w_{xxx} = f_0 + f_{1x}
\end{equation}
with initial and boundary conditions \eqref{4.23}.
The hypothesis of Theorem~\ref{T3.1} for this problem is satisfied and, therefore, the function
$w= S(u_0,\mu_0,\nu_0,\nu^*,f_0,f_1) \in X^k(Q_T)$ exists. Moreover, 
since $\partial_t^m w(0,x) = \widetilde\Phi_m(0;u_0,f_0+f_{1x})$ for $0\leq m\leq k$, conditions \eqref{4.42} ensure that $\varphi_0^{(m)}(0) = q^{(m)}(0;w,\omega_0)$ and, therefore, $\widetilde\varphi_0 \equiv \varphi_0 - Q(\omega_0)w \in \widetilde H^k(0,T)$.
As a result, $\Gamma_3 \widetilde\varphi_0 \in 
\widetilde H^k(0,T)$ is defined and so formulas \eqref{4.44} define the functions $\nu_1, u$ from the corresponding spaces. 

The end of the proof is similar to Theorem~\ref{T4.1}.
\end{proof}

\section{Nonlinear results}\label{S5}

\begin{proof}[Proof of Theorem \ref{T1.1}]
For $u_0 \in H^{3k}(I)$, $h_0\in M(Q_T)$ introduce a set
\begin{equation}\label{5.1}
X^k_{u_0,h_0}(Q_T) = \{u\in X^k(Q_T): \partial_t^m u\big|_{t=0} = \Phi_m(\cdot;u_0,h_0),\quad 0\leq m\leq k\}.
\end{equation}
For any $v\in X^k_{u_0,h_0}(Q_T)$ consider a map 
\begin{equation}\label{5.2}
u =\Theta_1 v \equiv  S(u_0,\mu_0,\nu_0,\nu_1,h_0+F h,-g(v)),
\end{equation}
where
\begin{gather}\label{5.3}
(F,\nu) = \Gamma_1\bigl(\varphi_1- Q(\omega_1)w, \varphi_2- Q(\omega_2)w \bigr), \\
\label{5.4}
w= S(u_0,\mu_0,\nu_0,\nu^*,h_0,- g(v)),\quad
\nu_1 = \nu + \nu^*,
\end{gather}
$\nu^*$ is the Taylor polynomial of the degree $k-1$ with the center $0$, satisfying    
$(\nu^*)^{(m)}(0) = \Phi'_m(R;u_0,h_0)$ for $0\leq m\leq k-1$, and the map $\Gamma_1$ is defined in Lemma~\ref{L4.1}.

Inequalities \eqref{2.4} and \eqref{2.7} yield that $(g(v))_x \in M_1^k(Q_T)$ and
\begin{equation}\label{5.5}
\|(g(v))_x\|_{M_1^k(Q_T)} \leq T^{1/2} \beta(\|v\|_{X^k(Q_T)}) \|v\|^2_{X^k(Q_T)} + \lambda(\rho(v)) \rho(v).
\end{equation}
In turn, inequality \eqref{2.10} yield that
\begin{equation}\label{5.6}
\rho(v) = \sum\limits_{m=0}^{k-1} \|\Phi_m(\cdot;u_0,h_0)\|_{H^{3(k-m)}(I)} \leq \sigma(\varkappa(u_0,h_0)) \varkappa(u_0,h_0).
\end{equation}
As a result, 
\begin{multline}\label{5.7}
\|(g(v))_x\|_{M_1^k(Q_T)} \\ \leq T^{1/2} \beta(\|v\|_{X^k(Q_T)}) \|v\|^2_{X^k(Q_T)} + \lambda\bigl(\sigma(\varkappa(u_0,h_0))\bigr) \sigma(\varkappa(u_0,h_0)) \varkappa(u_0,h_0).
\end{multline}

Next, equalities \eqref{1.10} and \eqref{3.4} ensure that 
\begin{equation}\label{5.8}
\Phi_m(x;u_0,h_0) \equiv \widetilde\Phi_m(x;u_0,h_0 - (g(v))_x), \quad 0\leq m \leq k. 
\end{equation}
In particular, it means that $\mu_0^{(m)}(0) = \widetilde\Phi_m(0;u_0,h_0-(g(v))_x)$, $\nu_0^{(m)}(0) = \widetilde\Phi_m(R;u_0,h_0-(g(v))_x)$, 
$(\nu^*)^{(m)}(0) = \widetilde\Phi'_m(R;u_0,h_0-(g(v))_x)$ for $0\leq m\leq k-1$, and, therefore, by Theorem~\ref{T3.1} the function $w\in X^k_{u_0,h_0}(Q_T)$, defined by formula \eqref{5.4}, exists and by Lemma~\ref{L3.1} $Q(\omega_j)w \in H^{(k+1)}(0,T)$.
Moreover, conditions \eqref{1.13} imply that 
\begin{equation}\label{5.9}
\varphi_j^{(m)}(0) = \bigl(Q(\omega_j)w\bigr)^{(m)}(0), \quad 0\leq m\leq k,
\end{equation}
therefore, $\varphi_j - Q(\omega_j)w \in \widetilde H^{k+1}(0,T)$ and the functions $(F,\nu) \in \bigl(\widetilde H^k(0,T)\bigr)^2$, defined by formula \eqref{5.3}, exist. 

For $0\leq m\leq k-1$ we have $\nu_1^{(m)}(0) = (\nu^*)^{(m)}(0)$ and since $\partial_t^m (F h)\big|_{t=0} \equiv 0$
\begin{equation}\label{5.10}
\Phi_m(x;u_0, h_0+F h) \equiv \Phi_m(x;u_0,h_0),
\end{equation}
therefore, the function $u\in X^k_{u_0,h_0}(Q_T)$, defined by formula \eqref{5.2}, exists, that is  $\Theta_1: X^k_{u_0,h_0}(Q_T) \to X^k_{u_0,h_0}(Q_T)$.

Next, obviously
\begin{equation}\label{5.11}
\|\nu^*\|_{H^k(0,T)} \leq c(T)\sum\limits_{m=0}^{k-1} |\Phi'_m(R;u_0,h_0)| \leq
c_1(T) \sum\limits_{m=0}^{k-1} \|\Phi_m(\cdot;u_0,h_0)\|_{H^{3(k-m)}(I)},
\end{equation}
by \eqref{3.6}
\begin{multline}\label{5.12}
\|w\|_{X^k(Q_T)} \leq c(T) \Bigl[ \|u_0\|_{H^{3k}(I)} + \|\mu_0\|_{H^{k+1/3}(0,T)} + 
\|\nu_0\|_{H^{k+1/3}(0,T)} + \|\nu^*\|_{H^k(0,T)} \\+
\|h_0\|_{M^k(Q_T)} + \|(g(v))_x\|_{M^k_1(Q_T)} \Bigr],
\end{multline}
by \eqref{3.12} 
\begin{multline}\label{5.13}
\bigl\|\bigl(Q(\omega_j)w\bigr)^{(k+1)}\bigr\|_{L_2(0,T)} \leq c(T)\Bigl[\|w\|_{X^k(Q_T)} + \|\mu_0\|_{H^k0,T)} + \|\nu_0\|_{H^k(0,T)}
\\+ \|h_0\|_{M^k(Q_T)} + \|(g(v))_x\|_{M^k_1(Q_T)} \Bigr],
\end{multline}
by Lemma~\ref{L4.1}
\begin{equation}\label{5.14}
\|\nu\|_{H^k(0,T)} + \|F\|_{\widetilde H^k(0,T)} \leq c(T) \sum\limits_{j=1}^2 \Bigl[\|\varphi_j^{(k+1)}\|_{L_2(0,T)} +
\bigl\|\bigl(Q(\omega_j)w\bigr)^{(k+1)}\bigr\|_{L_2(0,T)} \Bigr],
\end{equation}
since $\partial_t^m F\big|_{t=0} \equiv 0$ for $0\leq m\leq k-1$
\begin{equation}\label{5.15}
\|F h\|_{M^k(Q_T)} \leq c(T) \|F\|_{\widetilde H^k(0,T)}
\end{equation}
and again by \eqref{3.12}
\begin{multline}\label{5.16}
\|\Theta_1 v\|_{X^k(Q_T)} \leq c(T) \Bigl[ \|u_0\|_{H^{3k}(I)} + \|\mu_0\|_{H^{k+1/3}(0,T)} + \|\nu_0\|_{H^{k+1/3}(0,T)} +  \|\nu\|_{H^k(0,T)} \\+
\|\nu^*\|_{H^k(0,T)} + \|h_0\|_{M^k(Q_T)} + \|F h\|_{M^k(Q_T)} + \|(g(v))_x\|_{M^k_1(Q_T)} \Bigr].
\end{multline} 
Let
\begin{equation}\label{5.17}
\alpha(\delta) \equiv 1 +  \sigma(\delta) +\lambda\bigl(\sigma(\delta)\bigr) \sigma(\delta),
\end{equation}
then combining together \eqref{2.10}, \eqref{5.7}, \eqref{5.11}--\eqref{5.15}
we derive from \eqref{5.16}, \eqref{5.17} an inequality
\begin{equation}\label{5.18}
\|\Theta_1 v\|_{X^k(Q_T)} \leq c(T)\alpha(c_1)c_1 + c(T)T^{1/2} \beta(\|v\|_{X^k(Q_T)}) \|v\|^2_{X^k(Q_T)},
\end{equation}
where $c_1= c_1(u_0,\mu_0,\nu_0,h_0,\varphi_1,\varphi_2)$ is given by formula \eqref{1.15}.

Now let $v,\widetilde v \in X^k_{u_0,h_0}(Q_T)$, $\|v\|_{X^k(Q_T)}, 
\|\widetilde v\|_{X^k(Q_T)} \leq r$, then
\begin{gather}\label{5.19}
\Theta_1 v - \Theta_1 \widetilde v = S(0,0,0,\nu - 
\widetilde\nu, (F-\widetilde F)h, g(\widetilde v) - g(v)),\\
\label{5.20}
(F-\widetilde F, \nu-\widetilde\nu) = \Gamma_1\bigl(Q(\omega_1)(\widetilde w -w), Q(\omega_2)(\widetilde w -w)\bigr),\\
\label{5.21}
w-\widetilde w = S(0,0,0,0,0,g(\widetilde v)-g(v)).
\end{gather}
Here by \eqref{2.3}
\begin{equation}\label{5.22}
\bigl\|\bigl(g(v) - g(\widetilde v)\bigr)_x\bigr\|_{M_1^k(Q_T)} \leq T^{1/2}\beta(r)r \|u-v\|_{X^k(Q_T)},
\end{equation}
similarly to \eqref{5.12}
\begin{equation}\label{5.23}
\|w-\widetilde w\|_{X^k(Q_T)} \leq c(T) \bigl\|\bigl(g(v) - g(\widetilde v)\bigr)_x\bigr\|_{M_1^k(Q_T)},
\end{equation}
similarly to \eqref{5.13}
\begin{multline}\label{5.24}
\bigl\|\bigl(Q(\omega_j)(w-\widetilde w)\bigr)^{(k+1)}\bigr\|_{L_2(0,T)} \\ \leq c(T)\Bigl[\|w-\widetilde w\|_{X^k(Q_T)} + 
\bigl\|\bigl(g(v) - g(\widetilde v)\bigr)_x\bigr\|_{M_1^k(Q_T)}\Bigr],
\end{multline}
similarly to \eqref{5.14}
\begin{equation}\label{5.25}
\|\nu-\widetilde\nu\|_{W_2^k(0,T)} + \|F-\widetilde F\|_{\widetilde W_2^k(0,T)} \leq c(T) \sum\limits_{j=1}^2 
\bigl\|\bigl(Q(\omega_j)(w-\widetilde w)\bigr)^{(k+1)}\bigr\|_{L_2(0,T)},
\end{equation}
similarly to \eqref{5.16}
\begin{multline}\label{5.26}
\|\Theta_1 v - \Theta_1\widetilde v\|_{X^k(Q_T)} \\ \leq c(T) \Bigl[\|\nu-\widetilde\nu\|_{H^k(0,T)} + \|(F-\widetilde F) h\|_{M^k(Q_T)} + 
\bigl\|\bigl(g(v) - g(\widetilde v)\bigr)_x\bigr\|_{M_1^k(Q_T)} \Bigr]
\end{multline}
and, therefore, for the same constant $c(T)$ as in \eqref{5.18}
\begin{equation}\label{5.27}
\|\Theta_1 v -\Theta_1\widetilde v\|_{X^k(Q_T)} \leq c(T)T^{1/2} \beta(r)r \|v-\widetilde v\|_{X^k(Q_T)}.
\end{equation}

Now assume that $\delta>0$ is fixed. Choose $T_0>0$ and $r>0$ such that
\begin{equation}\label{5.28}
r \geq 2c(T_0)\alpha(\delta)\delta,\quad 2c(T_0) T_0^{1/2} \beta(r) r \leq 1
\end{equation}
(for example, one can first take arbitrary $T_0^*>0$, then choose $r>0$, satisfying $r\geq 
2c(T_0^*)\alpha(\delta)\delta$, and finally choose $T_0\in (0,T_0^*]$, satisfying $2c(T_0^*) T_0^{1/2} \beta(r) r \leq 1$).

If $T_0>0$ is fixed, first choose $r>0$, satisfying $2c(T_0) T_0^{1/2} \beta(r) r \leq 1$, and then $\delta>0$, satisfying $2c(T_0)\alpha(\delta)\delta \leq r$. 

Then in both cases, inequalities \eqref{5.28} are verified and it follows from estimates \eqref{5.18} and \eqref{5.27} that if $c_1\leq \delta$ and $T\in (0,T_0]$ the map $\Theta_1$ on the zero centered ball of the radius $r$ in the set $X_{u_0,f}$ is the contraction. The unique fixed point $u = S(u_0,\mu_0,\nu_0, \nu_1, h_0 +F h, -g(u))$ is the desired solution with the unique pair of functions $(F,h)$, given by formulas \eqref{5.3}, \eqref{5.4} (for $v\equiv u$).

To prove uniqueness in the whole space, assume that there are two different solutions $u, \widetilde u \in X^k(Q_T)$. Let $\|u\|_{X^k(Q_T)}, \|\widetilde u\|_{X^k(Q_T)} \leq r$, 
$T* = \max\{t: u(t,x) \equiv \widetilde u(t,x)\}$, then $T^*\in[0, T)$. Similarly to \eqref{5.19}--\eqref{5.20}
\begin{gather*}
u - \widetilde u = S(0,0,0,\nu - \widetilde\nu, (F-\widetilde F)h, g(\widetilde u) - g(u)),\\
(F-\widetilde F, \nu-\widetilde\nu) = \Gamma_1\bigl(Q(\omega_1)(\widetilde w -w), Q(\omega_2)(\widetilde w -w)\bigr),\\
w-\widetilde w = S(0,0,0,0,0,g(\widetilde u)-g(u)),
\end{gather*}
where similarly to \eqref{5.22}--\eqref{5.26} for any $T_0 \in (T^*, T]$ according to \eqref{2.3} 
\begin{multline*}
\|u - \widetilde u\|_{X^k(Q_{T_0})} \leq c(T) 
\bigl\|\bigl(g(u) - g(\widetilde u)\bigr)_x\bigr\|_{M_1^k(Q_{T_0})} \\= 
\sum\limits_{m=0}^k \bigl\|\partial_t^m \bigl(g(u) - 
g(\widetilde u)\bigr)\bigr\|_{L_2((T^*,T_0)\times I)} \leq 
c(T) (T_0 - T^*)^{1/2} \beta(r) r\|u - \widetilde u\|_{X^k(Q_{T_0})},
\end{multline*}
whence for $T_0$ sufficiently close to $T^*$ follows that $u(t,x) \equiv \widetilde u(t,x)$ if $t\in [0,T_0]$. Contradiction.

To prove continuous dependence, note that if
$c_1(u_0,\mu_0,\nu_0,h_0,\varphi_1,\varphi_2) \leq \delta$, $c_1(\widetilde u_0,\widetilde\mu_0,\widetilde\nu_0,\widetilde h_0,\widetilde\varphi_1,\widetilde\varphi_2) \leq \delta$,
the numbers $T_0$ and $r$ satisfy \eqref{5.28}, then the corresponding solutions $u, \widetilde u$ are the fixed points of the maps $\Theta_1$ for the corresponding input data and similarly to \eqref{5.18}, \eqref{5.27} with the use of \eqref{2.3}, \eqref{2.6}, \eqref{2.9}
\begin{multline}\label{5.29}
\|u - \widetilde u\|_{X^k(Q_{T})} \leq c(T) \alpha(\delta) \Bigl[ 
\|u_0 - \widetilde u_0\|_{H^{3k}(I)} +
\|\mu_0 - \widetilde\mu_0\|_{H^{k+1/3}(0,T)} \\+
\|\nu_0 - \widetilde\nu_0\|_{H^{k+1/3}(0,T)} +
\|h_0 - \widetilde h_0\|_{M^k(Q_T)} +
\|\varphi_1^{(k+1)} - \widetilde\varphi_1^{(k+1)}\|_{L_2(0,T)} \\+
\|\varphi_2^{(k+1)} - \widetilde\varphi_2^{(k+1)}\|_{L_2(0,T)} \Bigr] +
c(T) T^{1/2} \beta(r) r\|u - \widetilde u\|_{X^k(Q_{T})},
\end{multline}
whence the desired assertion follows.
\end{proof}

\begin{proof}[Proof of Theorem \ref{T1.2}]
For any $v\in X^k_{u_0,h_0}(Q_T)$ consider a map 
\begin{equation}\label{5.30}
u =\Theta_2 v \equiv  S(u_0,\mu_0,\nu_0,\nu_1,h_0+F h,-g(v)),
\end{equation}
where
\begin{equation}\label{5.31}
F = \Gamma_2\bigl(\varphi_0- Q(\omega_0)w \bigr), \quad
w= S(u_0,\mu_0,\nu_0,\nu_1,h_0,- g(v))
\end{equation}
and the map $\Gamma_2$ is defined in Lemma~\ref{L4.2}.

In comparison with the proof of the previous theorem here 
$\nu_1^{(m)}(0) = \widetilde\Phi'_m(R;u_0,h_0-(g(v))_x)$ for $0\leq m\leq k-1$, and, therefore, by Theorem~\ref{T3.1} the function $w\in X^k_{u_0,h_0}(Q_T)$, defined by formula \eqref{5.31}, exists. By Lemma~\ref{L3.1} $Q(\omega_0)w \in H^{(k+1)}(0,T)$ and
conditions \eqref{1.17} imply that 
\begin{equation}\label{5.32}
\varphi_0^{(m)}(0) = \bigl(Q(\omega_0)w\bigr)^{(m)}(0), \quad 0\leq m\leq k,
\end{equation}
therefore, $\varphi_0 - Q(\omega_0)w \in \widetilde H^{k+1}(0,T)$ and the function $F \in \widetilde H^k(0,T)$, defined by formula \eqref{5.31}, exists. Moreover, equality \eqref{5.10} is also valid and, therefore, the function $u\in X^k_{u_0,h_0}(Q_T)$, defined by formula \eqref{5.30}, exists, that is  $\Theta_2: X^k_{u_0,h_0}(Q_T) \to X^k_{u_0,h_0}(Q_T)$.

Next, similarly to \eqref{5.12}--\eqref{5.18} we find that
\begin{equation}\label{5.33}
\|\Theta_2 v\|_{X^k(Q_T)} \leq c(T)\alpha(c_2)c_2 + c(T)T^{1/2} \beta(\|v\|_{X^k(Q_T)}) \|v\|^2_{X^k(Q_T)},
\end{equation}
where $c_2= c_2(u_0,\mu_0,\nu_0,\nu_1,h_0,\varphi_0)$ is given by formula \eqref{1.19}, and similarly to \eqref{5.19}--\eqref{5.27} --- that
\begin{equation}\label{5.34}
\|\Theta_2 v -\Theta_2\widetilde v\|_{X^k(Q_T)} \leq c(T)T^{1/2} \beta(r)r \|v-\widetilde v\|_{X^k(Q_T)}.
\end{equation}
The end of the proof is the same as for the previous theorem. 
\end{proof}

\begin{proof}[Proof of Theorem \ref{T1.3}]
For any $v\in X^k_{u_0,h_0}(Q_T)$ consider a map 
\begin{equation}\label{5.35}
u =\Theta_3 v \equiv  S(u_0,\mu_0,\nu_0,\nu_1,f,-g(v)),
\end{equation}
where
\begin{equation}\label{5.36}
\nu = \Gamma_3\bigl(\varphi_0- Q(\omega_0)w\bigr), \quad
w= S(u_0,\mu_0,\nu_0,\nu^*,f,- g(v)),\quad
\nu_1 = \nu + \nu^*,
\end{equation}
$\nu^*$ is the Taylor polynomial of the degree $k-1$ with the center $0$, satisfying    
$(\nu^*)^{(m)}(0) = \Phi'_m(R;u_0,f)$ for $0\leq m\leq k-1$, and the map $\Gamma_3$ is defined in Lemma~\ref{L4.3}.

Here similarly to \eqref{5.8} 
\begin{equation}\label{5.37}
\Phi_m(x;u_0,f) \equiv \widetilde\Phi_m(x;u_0,f - (g(v))_x), \quad 0\leq m \leq k. 
\end{equation}
In particular, it means that $\mu_0^{(m)}(0) = \widetilde\Phi_m(0;u_0,f-(g(v))_x)$, $\nu_0^{(m)}(0) = \widetilde\Phi_m(R;u_0,f-(g(v))_x)$, 
$(\nu^*)^{(m)}(0) = \widetilde\Phi'_m(R;u_0,f-(g(v))_x)$ for $0\leq m\leq k-1$, and, therefore, by Theorem~\ref{T3.1} the function $w\in X^k_{u_0,h_0}(Q_T)$, defined by formula \eqref{5.36}, exists. By Lemma~\ref{L3.1} $Q(\omega_j)w \in H^{(k+1)}(0,T)$,
equalities \eqref{5.32} are satisfied,  
therefore, $\varphi_0 - Q(\omega_0)w \in \widetilde H^{k+1}(0,T)$ and the function $\nu \in \widetilde H^k(0,T)$, defined by formula \eqref{5.36}, exists. As a result,
the function $u\in X^k_{u_0,h_0}(Q_T)$, defined by formula \eqref{5.35}, exists, that is  $\Theta_3: X^k_{u_0,h_0}(Q_T) \to X^k_{u_0,h_0}(Q_T)$.

Then similarly to \eqref{5.11}-- \eqref{5.27} we derive estimates
\begin{gather}\label{5.38}
\|\Theta_3 v\|_{X^k(Q_T)} \leq c(T)\alpha(c_3)c_3 + c(T)T^{1/2} \beta(\|v\|_{X^k(Q_T)}) \|v\|^2_{X^k(Q_T)},\\
\label{5.39}
\|\Theta_3 v -\Theta_3\widetilde v\|_{X^k(Q_T)} \leq c(T)T^{1/2} \beta(r)r \|v-\widetilde v\|_{X^k(Q_T)},
\end{gather}
where $c_3= c_3(u_0,\mu_0,\nu_0,f,\varphi_0)$ is given by formula \eqref{1.22}.
The end of the proof is the same as for the previous theorems. 
\end{proof}

\begin{proof}[Proof of Theorem \ref{T1.4}]
For any $v\in X^k_{u_0,h_0}(Q_T)$ consider a map 
\begin{equation}\label{5.40}
u =\Theta_4 v \equiv  S(u_0,\mu_0,\nu_0,\nu_1,f,-g(v)).
\end{equation}

Here equalities \eqref{5.37} hold  
and, therefore, by Theorem~\ref{T3.1} the function $u\in X^k_{u_0,h_0}(Q_T)$, defined by formula \eqref{5.40}, exists, that is  $\Theta_4: X^k_{u_0,h_0}(Q_T) \to X^k_{u_0,h_0}(Q_T)$.

Then similarly to \eqref{5.38}, \eqref{5.39} we derive estimates
\begin{gather}\label{5.41}
\|\Theta_4 v\|_{X^k(Q_T)} \leq c(T)\alpha(c_4)c_4 + c(T)T^{1/2} \beta(\|v\|_{X^k(Q_T)}) \|v\|^2_{X^k(Q_T)},\\
\label{5.42}
\|\Theta_4 v -\Theta_3\widetilde v\|_{X^k(Q_T)} \leq c(T)T^{1/2} \beta(r)r \|v-\widetilde v\|_{X^k(Q_T)},
\end{gather}
where $c_4= c_4(u_0,\mu_0,\nu_0,\nu_1,f)$ is given by formula \eqref{1.24}.
The end of the proof is the same as for the previous theorems. 
\end{proof}

\begin{proof}[Proof of Corollary \ref{C1.1}]
For any $0\leq m \leq k$ we prove by induction with respect to $(k-m)$ that for $0\leq j\leq m$
\begin{equation}\label{5.43}
\partial_t^j u \in C([0,T];H^{3(k-m)}(I)) \cap L_2(0,T; H^{3(k-m)+1}(I)).
\end{equation}
For $m=k$ it is equivalent to the property $u\in X^k(Q_T)$. Let $m<k$. Then it follows from \eqref{5.43} for $(m+1)$ that
\begin{equation}\label{5.44}
\partial_t^j u \in C([0,T]; H^{3(k-m) -2}(I)), \quad 0\leq j \leq m,
\end{equation}
and, therefore, 
\begin{equation}\label{5.45}
\partial_t^j \partial_x^n u \in C(\overline Q_T),\quad 0\leq j \leq m,\ 0\leq n \leq 3(k-m-1).
\end{equation}
It follows from equality \eqref{1.1} that for $i=0,\, 1$
\begin{equation}\label{5.46}
\partial_t^j \partial_x^{3(k-m)+i} u= \partial_t^j\partial_x^{3(k-m-1)+i}
\bigl(f - u_t -b u_x - (g(u))_x\bigr).
\end{equation}
Then with the use of \eqref{5.43} for $(m+1)$, \eqref{5.44} and for the nonlinear term \eqref{5.45} and equality \eqref{1.10} (here the argument is similar to the one in the proof of Lemma~\ref{L2.4}) we find that for $i=0$ the right-hand side of \eqref{5.46} belongs to $C([0,T];L_2(I))$. As a result, 
\begin{equation}\label{5.47}
\partial_t^j u \in C([0,T];H^{3(k-m)}(I)).
\end{equation}
Returning to equality \eqref{5.46} for $i=1$ we find with the use of \eqref{5.47} that its right-hand side belongs to $L_2(Q_T)$ and, therefore, $\partial_t^j u \in L_2(0,T; H^{3(k-m)+1}(I))$.
\end{proof}




\end{document}